\documentclass[12pt]{article}

\title{Linear response for the dynamic Laplacian \\ and finite-time coherent sets}

\author{Fadi Antown\footnote{School of Mathematics and Statistics, University of New South Wales, Sydney NSW 2052, Australia.}, Gary Froyland\footnotemark[1] \ and Oliver Junge\footnote{Department of Mathematics,
Technical University of Munich, 85747 Garching, Germany.}}

\date{\today\\[5mm]
}
\usepackage{fullpage,amsmath,amssymb,amsthm,xcolor, soul,graphicx,bm}
\usepackage{float}
\usepackage[margin=0.93in]{geometry}
\usepackage[colorlinks]{hyperref}
\definecolor{winered}{rgb}{0.7,0,0}
\hypersetup{colorlinks=true,allcolors=winered}
\usepackage{bbm}

\newtheorem{prop}{Proposition}
\newtheorem{theorem}{Theorem}

\newtheorem{lemma}{Lemma}

\newtheorem{remark}{Remark}

\def\R{\mathbb{R}}
\def\N{\mathbb{N}}

\def\cB{\mathcal{B}}
\newcommand{\sym}{\text{\upshape sym}}
\newcommand{\divg}{\mathop{\rm div}}
\newcommand{\spn}{\mathop{\rm span}}
\newcommand{\eps}{\varepsilon}

\begin{document}

\maketitle
\begin{abstract}
Finite-time coherent sets represent minimally mixing objects in general nonlinear dynamics, and are spatially mobile features that are the most predictable in the medium term.  When the dynamical system is subjected to small parameter change, one can ask about the rate of change of (i) the location and shape of the coherent sets, and (ii) the mixing properties (how much more or less mixing), with respect to the parameter. We answer these questions by developing linear response theory for the eigenfunctions of the dynamic Laplace operator, from which one readily obtains the linear response of the corresponding coherent sets. We construct efficient numerical methods based on a recent finite-element approach and provide numerical examples.
\end{abstract}
\section{Introduction}

Finite-time coherent sets \cite{froyland2010transport,froyland2013analytic,froyland2015dynamic} are regions in the compact phase space of a nonlinear dynamical system that minimally mix over a finite time duration, and therefore play an important role in the analysis of how material objects are transported in fluids.
Spectral methods for identifying finite-time coherent sets were developed in \cite{froyland2010transport,froyland2013analytic} directly from transfer operators, and later in \cite{froyland2015dynamic,froyland2017dynamic} using the dynamic Laplacian, which was derived as a zero-diffusion limit of the transfer operator construction in \cite{froyland2013analytic}.
Various implementations of these two approaches and related methods may be found in \cite{mabollt13,froylandpadberg15,froylandjunge15,williams15,denner16,hadjighasem16, banischkoltai17,karrasch17,FrJu18,fackeldey19,FKS20}.
In the present paper we use the approach of \cite{froyland2015dynamic,froyland2017dynamic}, which defines finite-time coherent sets through the notion of dynamic isoperimetry:  those sets whose boundary size relative to volume remains small under the finite-time dynamics.
These persistently small boundaries represent evolving fluid interfaces across which there is minimal mixing.
The key technology for finding these coherent sets is the dynamic Laplace operator, defined in \cite{froyland2015dynamic};  the leading eigenfunctions of this operator encode the finite-time coherent sets.
The dynamic Laplacian and its eigenfunctions may be efficiently approximated using a specialised finite element method \cite{FrJu18}, and individual coherent sets may be automatically separated using algorithms such as SEBA \cite{froyland2019sparse}.

Throughout, we will represent the finite-time dynamics by a single application of a transformation $T_0:\Omega\to T_0(\Omega)$, $\Omega\subset\R^n$ compact;  $T_0$ may arise, for example, as a flow map of some nonautonomous ordinary differential equation.
The question we investigate in this work is \emph{how coherent sets behave under perturbation of  the dynamics}.
For some $\eps_0>0$ we consider a family of maps $T_\eps$, $\eps\in(-\eps_0,\eps_0)$, where we think of $T_0$ as governing the original, unperturbed dynamics.
As $\eps$ is varied from zero, the dynamic Laplacian corresponding to $T_\eps$, its eigenfunctions, and the corresponding coherent sets, all vary from those objects computed with $T_0$.
Given sufficient regularity of $\eps\mapsto T_\eps$ we may hope for some regular dependence of coherent sets on $\eps$.
The notion of coherent sets has found application in fluid flows from the laboratory scale to the planetary scale, and the dynamic Laplacian has proven to be an efficient way of extracting coherent objects (such as the Gulf Stream and ocean eddies \cite{froyland2019sparse}).
In the context of system perturbations due to climate change, an important step in quantifying potential impacts would be the prediction of responses of coherent geophysical features.

The approach we take is based on the ideas of linear response \cite{ruellemap,baladiicm} in dynamical systems.
Linear response is classically concerned with the derivatives of physical invariant measures $\mu_\eps$ of autonomous maps $T_\eps$ with respect to the parameter $\eps$.
The physical invariant measure $\mu_\eps$ is the leading eigenfunction (or eigendistribution) of the transfer operator $\mathcal{L}_\eps$ for $T_\eps$, and formulae for $\frac{\partial}{\partial\eps}\mu_\eps$ involve $\frac{\partial}{\partial\eps}\mathcal{L}_\eps$.
In order for physical invariant measures to exist, usually some hyperbolicity of the dynamics is required. For Anosov maps (and more general dynamical systems, like Axiom A), the
differentiability (which include the linear response) of the eigendata of the transfer operator
associated to these dynamical systems have been obtained \cite{gouezel06,gouezel08}; linear response results are also available for uniformly hyperbolic flows \cite{butterley,ruelleflow}.
Aside from smooth dynamics, linear response has also been treated for unimodal maps \cite{baladismania} and intermittent maps \cite{baladitodd};
there are also results for the existence of linear response for
stochastic dynamical systems
\cite{hairermajda,galatolo2019linear,bahsoun2020linear}.
Linear response is not guaranteed; see e.g.\ \cite{baladiicm} for details on failure of linear response.

In the present paper, we wish to follow this \emph{general} notion of linear response, namely computing the derivative of a quantity with respect to a parameter.
Specifically, we replace a family of transfer operators $\mathcal{L}_\eps$ with a family of dynamic Laplace operators arising from a family of diffeomorphisms $T_\eps$.
This new linear response problem requires a very different functional analytic setup and has rather well-behaved responses to perturbations.
We prove that differentiability of $\eps\mapsto T_\eps$ implies differentiability of the dynamic Laplacian.
Further, if a particular eigenvalue is algebraically simple, this eigenvalue and the associated eigenfunction have a linear response (are differentiable with respect to $\eps$).
We obtain a formula for the derivative of the eigenvalues with respect to the parameter;  the derivative of the second eigenvalue quantifies the instantaneous rate of change of global mixing as the parameter is varied.
We then derive a formula for the linear response of the eigenfunctions;
the derivatives of the dominant eigenfunctions of the dynamic Laplacian immediately yield derivatives of the corresponding finite-time coherent sets.
Building on the finite-element method (FEM) based approaches in \cite{FrJu18} we develop
numerical schemes for numerically computing these linear responses, and illustrate these schemes on the standard map and the Meiss-Mosovsky map.
In addition to computing the response of coherent sets, we observe that our first-order approximations of the perturbed eigenvectors, computed using linear response, produce coherent sets that are rather close to the true coherent sets, even for large extrapolation values.

An outline of the paper is as follows: in Section \ref{sec:intro-pert-setup} we introduce differentiability hypotheses on the dynamics.
In Section
\ref{sec:dynamic_isoperimetry} we define the dynamic Laplacian,
coherent sets, and our linear response problem.
Section \ref{sec:exi-lin-resp} contains the
proof of the weak differentiability of the dynamic Laplacian with respect to the perturbing parameter, and the proof of the existence of linear response of eigenvectors.
In Section \ref{section-theo-comp-u'} we derive a linear system whose solution provides the linear response, and in
Section \ref{section-numerical-computation} we develop two FEM-based
approaches to numerically solve this linear system and estimate the linear responses.
We conclude in Section \ref{sec:experiments}
with numerical demonstrations of the theory.

\section{Perturbations}\label{sec:intro-pert-setup}

Let $\Omega\subset\R^n$ be a compact, connected domain with smooth boundary.  We consider a family of maps $T_\eps:\Omega\rightarrow T_\eps(\Omega)$, $\eps\in (-\eps_0,\eps_0)$, where $T_0$ represents the original, unperturbed dynamics. For simplicity, we assume that $T_\eps$ is volume-preserving, and consider a single application of $T_\eps$. The methods we propose are easily extendable to non-volume-preserving $T_\eps$, curved manifolds \cite{froyland2017dynamic}, and multiple applications of $T_\eps$ (see \cite{froyland2015dynamic}).

Special families we have in mind are:
\begin{enumerate}
\item $T_\eps$ is given by the flow map $\varphi_\eps^{t_0,t_1}$ of some ordinary differential equation
\begin{equation*}
\partial_t x =v(t,x,\eps),
\end{equation*}
where the vector field $v$ depends on a parameter $\eps$ and $t_0,t_1\in\R$ are chosen such that the flow map is defined for all $x$.  Under appropriate assumptions on $v$ we have $T_\eps = T_0 + \eps \dot{T}_0 + o(\eps)$, where $\dot{T}_0(x) := \partial_\eps\varphi_\eps^{t_0,t_1}(x)|_{\eps=0}$ and $\partial_\eps\varphi_\eps^{t_0,t}(x)|_{\eps=0}$ satisfies the variational equation
\begin{equation}\label{variational_Tdot}
\partial_t\partial_\eps\varphi_\eps^{t_0,t}(x)|_{\eps=0} = \partial_x v(t,T_0(x),0)\partial_\eps\varphi_\eps^{t_0,t}(x)|_{\eps=0} + \partial_\eps v(t,T_0(x),0).
\end{equation}

\item As a further specialisation of 1.\ we interpret the time $t$ itself as the parameter $\eps$, i.e.\ we consider
\[
\partial_t x = v(x,t)
\]
with the flow map $T_0=\varphi^{t_0,t_1}$. In this case we have that
\begin{eqnarray*}
T_\eps (x)
:=  \varphi^{t_0,t_1+\eps}(x)
=  \varphi^{t_0,t_1}(x)+\eps\partial_{t}\varphi^{t_0,t}|_{t=t_1}(x)+ o(\eps)
=  T_0(x) + \eps \dot{T}_0(x) + o(\eps),
\end{eqnarray*}
where $\dot{T}_0(\cdot) := {\partial_{t}}\varphi^{t_0,t}|_{t=t_1}(\cdot) = v(\cdot,t_1)$.
\end{enumerate}

The precise setting we consider is the following: Let Diff$^2(\Omega,\R^n)$ be the
space of $C^2$-diffeo\-morphisms from $\Omega$ to $\R^n$ which is endowed with the
$C^2$-norm
$$
\|f\|_{C^2(\Omega,\R^n)}
= \sum_{\alpha_j,|\alpha|\le 2}\max_{x\in\Omega}\bigg\|\frac{\partial^{|\alpha|}f}{\partial x^{\alpha_1}_{1}\dots\partial x^{\alpha_n}_{n}}(x)\bigg\|,
$$
where $\alpha = (\alpha_1,\dots,\alpha_n)\in \N_0^n, |\alpha| =\sum_{j=1}^n\alpha_j$.
We  assume that the map $\eps\mapsto T_\eps$ is $C^1$ from $(-\eps_0,\eps_0)\subset\R$ to $\text{Diff}^2(\Omega,\R^n)$. From Taylor's theorem (see \cite{lang2012real}, XIII \S 6)
for sufficiently small $\eps_0>0$, one has
\begin{equation}\label{eq:Taylor-Dynamics}
T_\eps = T_0+\eps \dot{T}_0 + R_\eps
\end{equation}
for  $\eps\in (-\eps_0,\eps_0)$, where $\dot{T}_0, R_\eps\in \text{Diff}^2(\Omega,\R^n)$, and $\|R_\eps\|_{C^2(\Omega,\R^n)}=o(\eps)$. Since all $T_\eps$ are $C^2$ diffeomorphisms we have that for any
$\eps\in (-\eps_0,\eps_0)$ the maps $DT_\eps$ and $DT^{-1}_\eps$ are in $ C^1(\Omega, \mathcal{B}(\R^n))$, where $\mathcal{B}(X)$ is the space of bounded linear maps from some Banach space $X$ into itself.

\section{The dynamic Laplacian}
\label{sec:dynamic_isoperimetry}

We are ultimately interested in analysing the response of coherent sets to perturbations of the dynamics.  As coherent sets can be detected via level sets of leading nontrivial eigenfunctions of the dynamic Laplace operator, we need to understand how these eigenfunctions respond to perturbations in the dynamics, i.e.\ how they change with $\eps$.

Following \cite{froyland2015dynamic}, when dividing a manifold $\Omega$ into two coherent sets, one seeks a dynamically minimal interface $\Gamma$ disconnecting $\Omega$;  the interface $\Gamma$ forms the shared boundary of the two coherent sets.
More precisely, if $\Gamma\subset\Omega$ is a codimension-1 submanifold disconnecting $\Omega$ into $\Omega_1$ and $\Omega_2$, we compute the dynamic Cheeger value of $\Gamma$:
\begin{equation}
\label{cheeger}
\mathbf{h}(\Gamma) := \frac{\frac{1}{2} (\ell_{n-1}(\Gamma)+\ell_{n-1}(T_\eps(\Gamma)))}{\min\{\ell(\Omega_1),\ell(\Omega_2)\}},
\end{equation}
where $\ell_{n-1}$ is the induced $n-1$-dimensional volume and $\ell$ is the $n-$dimensional volume. We
seek the minimising $\Gamma$ to obtain the dynamic Cheeger constant \cite{froyland2015dynamic}:
\begin{equation}
\mathbf{h}:= \min\{\mathbf{h}(\Gamma):\Gamma\text{ is a }C^\infty\text{ codimension 1}\text{ manifold disconnecting }\Omega\}.
\end{equation}
In the case where we do not wish the interface $\Gamma$ to intersect the boundary of $\Omega$ we can alternatively consider a Dirichlet dynamic Cheeger constant;  see \S2.2 \cite{FrJu18}.
These two options are summarised in \cite[Figures 2 and 3]{FrJu18}, respectively.

A minimizing $\Gamma$ can be linked to level sets of eigenfunctions of a \emph{dynamic Laplace operator}, see \cite{froyland2015dynamic,froyland2017dynamic,FrJu18}. Denote the pushforward resp.\ pullback of a function $f:\Omega\to\mathbb{R}$ by $T_{\eps,*}f := f\circ T_\eps^{-1}$ resp.\ $T_\eps^*f := f\circ T_\eps$ and let  $\Delta_\Omega$ be the Laplace operator on $\Omega$.
The dynamic Laplace operator \cite{froyland2015dynamic} is
\begin{equation}
\label{DLdef}
\Delta^D_\eps := \frac{1}{2}\left(\Delta_\Omega + T_\eps^*\Delta_{T_\eps(\Omega)} T_{\eps,*}\right).
\end{equation}
Define the matrix-valued function $A_\eps\in C^1(\Omega, \mathcal{B}(\R^n))$ by
\begin{equation}
\label{coeffmatrix}
A_\eps:= \frac{1}{2}\left(I+(DT_\eps^\top DT_\eps)^{-1}\right).
\end{equation}
We are interested in the eigenproblem
\begin{equation}\label{eq:eign-prob-coh-set-1}
\Delta_\eps^D u_\eps = \lambda_\eps u_\eps\qquad\text{ on int}(\Omega),
\end{equation}
with homogeneous Neumann (resp.\ Dirichlet) boundary conditions
\begin{equation}\label{eq:eign-prob-coh-set-2}
\nabla u_\eps\bullet A_\eps\nu = 0\qquad(\text{resp.}\ u_\eps=0) \qquad\text{ on }\partial\Omega
\end{equation}
($\nu$ denotes the outer normal on $\Omega$). The spectral properties of the family $\Delta^D_\eps$ are given by Theorem 4.1 \cite{froyland2015dynamic}.
A discussion of the interpretation of the (natural) Neumann boundary conditions is given immediately after Theorem 3.2 \cite{froyland2015dynamic};  the Dirichlet boundary condition case is developed in \cite{FrJu18}.
Throughout the paper, we will assume that all eigenvalues of $\Delta^D_\eps$ are algebraically simple.

The weak form of the eigenvalue problem \eqref{eq:eign-prob-coh-set-1}--\eqref{eq:eign-prob-coh-set-2} is given by
\begin{equation}
\label{Eigen-Prob-Weak-form}
-\frac{1}{2}\left(\int_{\Omega}  \nabla  u_\eps\bullet\nabla \varphi \; d\ell + \int_{T_\eps(\Omega)}  \nabla (T_{\eps,*} u_\eps)\bullet\nabla (T_{\eps,*}\varphi) \; d\ell\right) = \lambda_\eps\int_\Omega u_\eps\varphi \; d\ell \qquad\forall \varphi\in H,
\end{equation}
where $H$ denotes $H^1(\Omega)$ in the case of Neumann and $H^1_0(\Omega)$ in the case of homogenous Dirichlet boundary conditions.  Note that if we let $\varphi = u_\eps$ in \eqref{Eigen-Prob-Weak-form}, all integrals are positive; thus, the eigenvalues $\lambda_\eps$ are negative (or $0$). Note further that
\begin{equation*}
\begin{aligned}
-\int_{T_\eps(\Omega)}  \nabla (T_{\eps,*}u_\eps)\bullet\nabla (T_{\eps,*}\varphi)\; d\ell
&= -\int_\Omega  (DT_\eps^\top DT_\eps)^{-1}\nabla u_\eps\bullet \nabla \varphi\; d\ell,
\end{aligned}
\end{equation*}
so that (\ref{Eigen-Prob-Weak-form}) can be written as
\begin{equation}
\label{Eigen-Prob-Weak-form-2}
- a_\eps(u_\eps,\varphi) = \lambda_\eps \langle u_\eps,\varphi \rangle \qquad\forall \varphi\in H.
\end{equation}
with the bilinear form $a_\eps(u_\eps,\varphi) = \int_{\Omega} A_\eps \nabla  u_\eps\bullet\nabla \varphi \; d\ell$, and $\langle\cdot,\cdot\rangle$ the $L^2$ scalar product on $\Omega$.
By the above considerations we may also write $\Delta^D_\eps=\sum_{i,j=1}^n \partial_jA_{\eps,ij}\partial_i$.

\section{Existence of a Linear Response}\label{sec:exi-lin-resp}

Throughout we assume that $u_\eps$ is scaled so that $\|u_\eps\|=1$, where $\|\cdot\|$ is the $L^2(\Omega,\ell)$ norm.
In order to answer the question of how coherent sets of $T_\eps$ depend on $\eps$, we are going to show that the map $\eps\mapsto u_\eps$ is differentiable at $0$ as a map from $(-\eps_0,\eps_0)$ to
$H$ and devise a method for computing the \textit{linear response}
\[
\dot{u}_0 := \frac{d}{d\eps} u_\eps|_{\eps=0}.
\]

We begin with a lemma about the regularity of the coefficient function $A_\eps$ of the dynamic Laplace operator.
In Proposition \ref{thm:deriv-dyna-Lap} we show that we can differentiate (in a weak sense) the map $\eps\mapsto\Delta^D_\eps$.
Finally, we apply a general regularity theorem for the spectral
data of elliptic operators to obtain the differentiability of the maps $\eps\mapsto\lambda_\eps$ and $\eps\mapsto u_\eps$.

Let $Q^\sym = \frac12(Q+Q^\top)$ denote the symmetric part of a matrix $Q$.
\begin{lemma}\label{lem:deriv-coeif-para-Lap}
The matrix-valued function $\dot{A}_0\in C^1(\Omega,\cB(\R^n))$ given by \begin{equation}\label{eq:derv-mat-coef}
\dot{A}_0 = -\left((DT_0)^{-1}(D\dot{T}_0) (DT_0)^{-1}(DT_0)^{-\top}\right)^\sym
\end{equation}
satisfies
\begin{equation}
\lim_{\eps\rightarrow 0}\bigg\|\frac{A_\eps - A_0}{\eps} -\dot{A}_0\bigg\|_{C^1(\Omega,\cB(\R^n))} = 0.
\end{equation}
\end{lemma}
\begin{proof}

We recall from \eqref{eq:Taylor-Dynamics} that for sufficiently small $\eps_0>0$,
we have that
$T_\eps = T_0+\eps \dot{T}_0 + R_\eps$ for $\eps\in(-\eps_0,\eps_0)$ with $\|R_\eps\|_{C^2(\Omega,\R^n)}=o(\eps)$, yielding
\begin{eqnarray}
   \label{inverse} (DT_\eps)^{-1} &=& \left(DT_0 + \eps D\dot{T}_0 + DR_\eps\right)^{-1}
= \left(\text{Id}+(DT_0)^{-1}(\eps D\dot{T}_0 +DR_\eps)\right)^{-1}(DT_0)^{-1}.
\end{eqnarray}
Using the fact that $\|R_\eps\|_{C^2(\Omega,\R^n)}=o(\eps)$, we have that
$\|DR_\eps\|_{C^1(\Omega,\mathcal{B}(\R^n))} = o(\eps)$ and so there exists
$C<\infty$, that is independent of $\eps$, such that
\begin{align*}
    &\|(DT_0)^{-1}(\eps D\dot{T}_0+DR_\eps)\|_{C^1(\Omega,\mathcal{B}(\R^n))}\\
&\le \|(DT_0)^{-1}\|_{C^1(\Omega,\mathcal{B}(\R^n))} \left(|\eps|\|D\dot{T}_0\|%
_{C^1(\Omega,\mathcal{B}(\R^n))} + \|DR_\eps\|_{C^1(\Omega,%
\mathcal{B}(\R^n))}\right)\\
&\le|\eps| \|(DT_0)^{-1}\|_{C^1(\Omega,\mathcal{B}(\R^n))} \left(\|D\dot{T}_0\|%
_{C^1(\Omega,\mathcal{B}(\R^n))} + C\right).
\end{align*}
Choosing $\eps$ small enough to satisfy
$$|\eps|< \left(\|(DT_0)^{-1}\|_{C^1(\Omega,\mathcal{B}(\R^n))} \left(\|D\dot{T}_0\|%
_{C^1(\Omega,\mathcal{B}(\R^n))} + C\right)\right)^{-1}$$
we get
$\|(DT_0)^{-1}(\eps D\dot{T}_0+DR_\eps)\|_{C^1(\Omega,\mathcal{B}(\R^n))} <1$.
We can now use the Neumann series representation for the RHS of (\ref{inverse}) to obtain
\begin{eqnarray}
  \nonumber  (DT_\eps)^{-1}&=& \left(\text{Id}-(DT_0)^{-1}(\eps D\dot{T}_0 +DR_\eps)+\left((DT_0)^{-1}(\eps D\dot{T}_0 +DR_\eps)\right)^2 - \cdots \right)(DT_0)^{-1}\\
\nonumber &=& \left(\text{Id} - \eps (DT_0)^{-1}D\dot{T}_0 + \widehat{R}_\eps\right)(DT_0)^{-1}\\
\label{invexp}&=& (DT_0)^{-1} -\eps (DT_0)^{-1}(D\dot{T}_0)(DT_0)^{-1} + \widehat{R}_\eps(DT_0)^{-1},
\end{eqnarray}
where $\widehat{R}_\eps = (DT_0)^{-1}DR_\eps + \sum_{i\ge 2} (-1)^i%
\left((DT_0)^{-1}(\eps D\dot{T}_0 +DR_\eps)\right)^i$. Noting that
\begin{align*}
    \| \widehat{R}_\eps\|_{C^1(\Omega,\mathcal{B}(\R^n))}&\le \|(DT_0)^{-1}\|_{C^1(\Omega,\mathcal{B}(\R^n))} \|DR_\eps\|%
_{C^1(\Omega,\mathcal{B}(\R^n))}\\
&+ \sum_{i\ge 2}%
\|(DT_0)^{-1}\|_{C^1(\Omega,\mathcal{B}(\R^n))}^i%
\left(\eps\|D\dot{T}_0\|_{C^1(\Omega,\mathcal{B}(\R^n))}%
+\|DR_\eps\|_{C^1(\Omega,\mathcal{B}(\R^n))}\right)^i,
\end{align*}
and using the fact that $\|DR_\eps\|_{C^1(\Omega,\mathcal{B}(\R^n))} = o(\eps)$, we have
that $\| \widehat{R}_\eps\|_{C^1(\Omega,\mathcal{B}(\R^n))} = o(\eps)$.
Hence, using (\ref{invexp}) we get
\begin{align*}
    &(DT_\eps)^{-1}(DT_\eps)^{-\top} \\
&=(DT_0)^{-1}(DT_0)^{-\top}\\
&\quad-\eps\left((DT_0)^{-1}(D\dot{T}_0)(DT_0)^{-1}(DT_0)^{-\top} + (DT_0)^{-1}(DT_0)^{-\top}(D\dot{T}_0)^\top(DT_0)^{-\top} \right) +\widetilde{R}_\eps\\
&= (DT_0)^{-1}(DT_0)^{-\top} +2 \eps \dot{A}_0 + \widetilde{R}_\eps,
\end{align*}
where $\|\widetilde{R}_\eps\|_{C^1(\Omega,\mathcal{B}(\R^n))} = o(\eps)$.
We conclude that $\|A_\eps- A_0 -\eps \dot{A}_0\|_{C^1(\Omega,\mathcal{B}(\R^n))} = o(\eps)$.
\end{proof}

We define $\dot\Delta_0^D:=\divg(\dot A_0\nabla)$ and consider the associated bilinear form
\begin{equation}
\label{eq:dotbilinform}
\dot a_0(\psi,\varphi) := \int_\Omega \dot A_0 \nabla\psi\bullet\nabla\varphi \;d\ell.
\end{equation}

\begin{prop}\label{thm:deriv-dyna-Lap}
The bilinear form (\ref{eq:dotbilinform}) is a weak derivative of the weak form of $\Delta_\eps^D$ at $\eps=0$ in the sense that for $\psi,\varphi\in H$
\begin{equation}
  \dot a_0(\psi,\varphi) = \lim_{\eps\rightarrow 0}  \frac{a_\eps(\psi,\varphi) - a_0(\psi,\varphi)}{\eps}.
\end{equation}
\end{prop}

\begin{proof}
We have
\[
a_\eps(\psi,\varphi) - a_0(\psi,\varphi) - \eps\dot a_0(\psi,\varphi) = \int_\Omega (A_\eps - A_0 -\eps\dot A_0)\nabla\psi\bullet\nabla\varphi\; d\ell.
\]
Lemma \ref{lem:deriv-coeif-para-Lap} yields $A_\eps = A_0 + \eps \dot{A}_0 + R_\eps$ with $\|R_\eps\|_{C^1(\Omega,\cB(\R^n))} = o(\eps)$. We therefore immediately get
\begin{equation}\label{eq:linear-order-weak-form-bound}
\begin{aligned}
\bigg|\int_\Omega (A_\eps - A_0 - \eps \dot{A}_0)\nabla \psi\bullet \nabla\varphi \;d\ell \bigg|
&= \bigg|\int_\Omega R_\eps\nabla \psi\bullet \nabla\varphi \;d\ell \bigg|\\
&\le \|R_\eps\|_{C^0(\Omega,\cB(\R^n))}\|\nabla \psi\|\|\nabla\varphi\| = o(\eps).
\end{aligned}
\end{equation}
\end{proof}

We now state a theorem concerning differentiability of the spectral data for the eigenproblem:
\begin{equation}\label{D-Eig-Prob}
\begin{aligned}
L_A u & = \lambda u 	\quad \text{in } \Omega, \\
    u & = 0 			\quad \text{on }\partial \Omega
\end{aligned}
\end{equation}
of some general uniformly elliptic second order differential operator $L_A = \sum_{i,j=1}^n \partial_j A_{ij}\partial_i$ with coefficients $A=(A_{ij})$. Let $\Lambda(L_A)\subset\mathbb{R}\times H$ be the set of eigenpairs $(\lambda, u)$ of $L_A$.

\begin{theorem}[\cite{haddad2015differentiability}]\label{Theorem-Had-Mon}
Let $O\subset\mathbb{R}^n$ be a bounded domain and $A_0\in C^{k}(\overline O,\R)^{n^2}$, $k\geq 1$,  the coefficients of the uniformly elliptic operator $L_{A_0}$. Let $(\lambda_0,u_0)\in\Lambda (L_{A_0})$ and assume $\lambda_0$ is algebraically simple. Then there exists a neighbourhood $\mathcal{U}\subset C^{k}(\overline O,\R)^{n^2}$ of $A_0$ and $C^{k}$-functions $\boldsymbol\lambda :\mathcal{U}\rightarrow\mathbb{R}$ and $\mathbf{u}:\mathcal{U}\rightarrow H_0^1(O)$ such that:
\begin{enumerate}
\item $\boldsymbol\lambda(A_0) = \lambda_0$ and $\mathbf{u}(A_0)=u_0$;
\item $(\boldsymbol\lambda(A),\mathbf{u}(A))\in \Lambda(L_{A})$ for every $A\in \mathcal{U}$.
\end{enumerate}
\end{theorem}

Let $(A_{\eps,ij})$ be the entries of $A_\eps$. Note that  since $A_\eps$ is in $C^1(\Omega,\mathcal{B}(\R^n))$, we have that $(A_{\eps,ij})\in C^1(\Omega,\R)^{n^2}$.
We note further that $L_{A_\eps}=\Delta_\eps^D$ is uniformly elliptic \cite{froyland2015dynamic,froyland2017dynamic} and so this theorem applies to the eigenproblem \eqref{eq:eign-prob-coh-set-1}--\eqref{eq:eign-prob-coh-set-2} setting $\overline O=\Omega$.  We note that the proof in \cite{haddad2015differentiability} does not make use of the assumption of zero Dirichlet boundary data and in fact also applies to the Neumann boundary case.

In the subsequent results, $H$ denotes $H^1(\Omega)$ in the case of homogeneous Neumann boundary conditions, and $H^1_0(\Omega)$ in the case of homogeneous Dirichlet boundary conditions.  The following theorem establishes the existence of derivatives of the maps $\eps\mapsto u_\eps$ and $\eps\mapsto \lambda_\eps$ from $(-\eps_0,\eps_0)$ to $H$.  Let $V_0:=$ span$\{u_0\}^\perp\subset H$.

\begin{theorem}\label{existence-derivative}
Let $\lambda_0$ be algebraically simple and $(\lambda_\eps,u_\eps)\in\Lambda(\Delta_\eps^D)$ for $\eps\in(-\eps_0,\eps_0)$. Then there exists a function
$\dot{u}_0\in H$ and $\dot{\lambda}_0\in\mathbb{R}$ such that
\begin{equation*}
\lim_{\eps\rightarrow 0}\bigg\|\frac{u_\eps - u_0}{\eps} - \dot{u}_0\bigg\|_{H} = 0
\quad\text{and}\quad
\lim_{\eps\rightarrow 0}\bigg|\frac{\lambda_\eps-\lambda_0}{\eps} - \dot{\lambda}_0\bigg| = 0.
\end{equation*}
Furthermore, $\dot u_0\in V_0$.
\end{theorem}

\begin{proof}
Let $\mathcal{U}\ni A_0$ be the neighborhood and $\mathbf{u}:\mathcal{U}\rightarrow H$ and $\boldsymbol\lambda:\mathcal{U}\rightarrow\mathbb{R}$ the maps  according to Theorem~\ref{Theorem-Had-Mon}.  Since these maps are $C^1$, there exist bounded linear maps $B_1: C^1(\Omega,\R)^{n^2}\rightarrow H$ and $B_2:C^1(\Omega,\R)^{n^2}\rightarrow\mathbb{R}$ satisfying
\begin{equation*}
\lim_{\|A_\eps - A_0\|_{C^1(\Omega,\R)^{n^2}}\rightarrow 0} \frac{\|\mathbf{u}(A_\eps)-\mathbf{u}(A_0) - B_1(A_\eps - A_0)\|_{H}}{\|A_\eps - A_0\|_{C^1(\Omega,\R)^{n^2}}} = 0
\end{equation*}
and
\begin{equation*}
\lim_{\|A_\eps - A_0\|_{C^1(\Omega,\R)^{n^2}}\rightarrow 0} \frac{|\boldsymbol\lambda(A_\eps)-\boldsymbol\lambda(A_0) - B_2(A_\eps - A_0)|}{\|A_\eps - A_0\|_{C^1(\Omega,\R)^{n^2}}} = 0.
\end{equation*}

Define $\dot{u}_0 := B_1(\dot{A}_0)\in H$ and ${\dot\lambda_0}:= B_2(\dot{A}_0)$. Using  $A_{\eps,ij} = A_{0,ij} + \eps \dot{A}_{0,ij} + r^\eps_{ij}$ with
$\|r^\eps_{ij}\|_{C^1(\Omega,\R)} = o(\eps)$,
the differentiability results follow.

In order to show that $\dot u_0\in V_0$, we note that for small $\eps$, we have  $u_\eps = u_0+\eps \dot u_0+g^\eps$, where $g^\eps \in H$ is such that $\|g^\eps\|_{H} =o(\eps)$. We therefore have
$$1=\langle u_\eps, u_\eps\rangle = \langle u_0,u_0\rangle + 2\eps\langle u_0,\dot u_0\rangle+ 2\langle u_0,g^\eps\rangle = 1+2\eps\langle u_0,\dot u_0\rangle+ 2\langle u_0,g^\eps\rangle;$$
thus, considering the leading term of order $\eps$ we see that $\langle u_0,\dot u_0\rangle=0$ and therefore $\dot u_0\in V_0$.
\end{proof}

\section{A formula for the linear response}
\label{section-theo-comp-u'}

We will now derive a linear system that yields the linear response $\dot u_0$ as its solution.  To this end, we first show that the (weak) derivative of the products $\eps\mapsto\lambda_\eps u_\eps$ and $\eps\mapsto\Delta_\eps^Du_\eps$ can be computed by the usual product rule.

\begin{lemma}
\label{lem:lin-sys-lin-resp1}
For $\varphi\in H$,
\begin{eqnarray}
\label{dlamu}\lim_{\eps\rightarrow 0} \left\langle  \frac{\lambda_\eps u_\eps - \lambda_0 u_0}{\eps}, \varphi \right\rangle & = & \left\langle  \lambda_0\dot{u}_0 + \dot{\lambda}_0 u_0,  \varphi \right\rangle \quad\text{and}\\
\label{dDelu}\lim_{\eps\rightarrow 0}  \frac{a_\eps(u_\eps,\varphi) - a_0(u_0,\varphi)}{\eps} & =  & a_0(\dot u_0,\varphi) + \dot a_0(u_0,\varphi) .
\end{eqnarray}
\end{lemma}
\begin{proof}
From Theorem \ref{existence-derivative} we have that $u_\eps = u_0+\eps \dot{u}_0 + g^\eps$ and
$\lambda_\eps = \lambda_0 +\eps \dot{\lambda}_0 + \mu^\eps$, where $\|g^\eps\|_{H}=o(\eps)$ and $|\mu^\eps| = o(\eps)$.
Thus,
$$
\lambda_\eps u_\eps = \lambda_0u_0+ \eps (\lambda_0 \dot{u}_0 + \dot{\lambda}_0u_0) + f^\eps
$$
with $\|f^\eps\| = o(\eps)$, so that $|\langle f^\eps,\varphi\rangle| \leq \|f^\eps\| \|\varphi\| = o(\eps)$ for each $\varphi\in H$.  This yields \eqref{dlamu}.

From Proposition~\ref{thm:deriv-dyna-Lap},
$a_\eps(\psi,\varphi) = a_0(\psi,\varphi)+\eps\dot a_o(\psi,\varphi)+o(\eps)$ for all $\psi,\varphi\in H$. Hence,
\begin{align*}
a_\eps(u_\eps,\varphi)
& = a_\eps(u_0+\eps\dot u_0+g^\eps) \\
& = a_\eps(u_0,\varphi) + \eps a_\eps(\dot u_0,\varphi) + o(\eps) \\
& = a_0(u_0,\varphi) + \eps \dot a_0(u_0,\varphi) + \eps a_0(\dot u_0,\varphi) + o(\eps),
\end{align*}
yielding \eqref{dDelu}.

\end{proof}

The following theorem establishes the existence of a unique solution of the linear system \eqref{eq:LR-nonweak-form} in a weak sense.

\begin{theorem}\label{th:lin-sys-lin-resp}
Let
$\dot{\lambda}_0$ and $\dot{u}_0$ be as in Theorem \ref{existence-derivative}. These linear responses $(\dot{u}_0,\dot{\lambda}_0)\in V_0\times \mathbb{R}$ are the unique solution to the equations:
\begin{equation}\label{eq:LR_weak}
  a_0(\dot u_0,\varphi) - \lambda_0\langle \dot u_0,\varphi\rangle = - \left(\dot a_0(u_0,\varphi) - \dot \lambda_0  \langle u_0,\varphi\rangle\right)
  \quad\text{for all } \varphi\in V_0.
\end{equation}
and
\begin{equation}\label{eq:LR_evalue}
  \dot{\lambda}_0 = \frac{\dot a_0(u_0, u_0)}{\|u_0\|^2}.
\end{equation}
\end{theorem}
\begin{proof}
We begin by showing that $\dot{\lambda}_0$ and $\dot{u}_0$ as in Theorem \ref{existence-derivative} solve (\ref{eq:LR_weak}) for all $\varphi\in V_0$ and (\ref{eq:LR_evalue}).
Subtract (\ref{dlamu}) from (\ref{dDelu});  we obtain 0 on the LHS because $u_\eps$ is the eigenfunction associated to the eigenvalue $\lambda_\eps$.
Rearranging the RHS we immediately obtain that (\ref{eq:LR_weak}) is satisfied for all $\varphi\in H$.

We now write $H=\spn\{u_0\}\oplus V_0$ and consider (\ref{eq:LR_weak}) for $\varphi$ according to this decomposition.
Substituting $\varphi = u_0$ into \eqref{eq:LR_weak} yields
\begin{equation}
  a_0(\dot u_0,u_0) - \lambda_0\langle \dot u_0, u_0\rangle = - \left(\dot a_0(u_0,u_0) - \dot \lambda_0  \langle u_0,u_0\rangle\right).
\end{equation}
The LHS is zero since $u_0$ is the eigenfunction with eigenvalue $\lambda_0$;  rearranging to solve for $\dot{\lambda}$ yields \eqref{eq:LR_evalue}.
Thus, (\ref{eq:LR_weak}) holding for all $\varphi\in H$ is equivalent to \eqref{eq:LR_weak}, \eqref{eq:LR_evalue}, proving the statement, except for uniqueness.

Suppose that there is another pair $(\dot{w},\dot{\nu})\in V_0\times\mathbb{R}$ satisfying (\ref{eq:LR_weak}) for all $\varphi\in H$.
Subtracting (\ref{eq:LR_weak}) with $(\dot{w},\dot{\nu})$ from (\ref{eq:LR_weak}) with $(\dot{u}_0,\dot{\lambda}_0)$ we obtain
\begin{equation}
  \label{diff}
  a_0(\dot{u}_0-\dot{w},\varphi) - \lambda_0\langle \dot{u}_0-\dot{w},\varphi \rangle
  = \dot{\lambda}_0 \langle u_0, \varphi\rangle
    - \dot{\nu} \langle u_0, \varphi\rangle
  \end{equation}
We again use the decomposition $H=\spn\{u_0\}\oplus V_0$. Substituting $\varphi=u_0$, and arguing as previously, we see that the LHS of (\ref{diff}) is zero and therefore that $\dot{\nu}=\dot{\lambda}_0$, i.e.
\[
a_0(\dot u_0 - \dot w, \varphi) - \lambda_0\langle \dot{u}_0-\dot{w},\varphi \rangle = 0 \quad \text{for all } \varphi \in H
\]
which implies that $\dot{u}_0-\dot{w}$ is a weak eigenfunction with eigenvalue $\lambda_0$.
Because $\lambda_0$ is simple, we must have $\dot{u}_0-\dot{w}\in \spn\{u_0\}$.
Recalling that $\dot{u}_0,\dot{w}\in V_0$, this implies that $\dot{u}_0-\dot{w}=0$.
Thus with $\dot{\lambda}_0$ as in (\ref{eq:LR_evalue}), there is a unique solution $\dot{u}_0$ to (\ref{eq:LR_weak}).
\end{proof}

We note that the strong form of \eqref{eq:LR_weak} is given by the equation
\begin{equation}
\label{eq:LR-nonweak-form}
	\begin{array}{ll}
    (\Delta_0^D - \lambda_0 I)\dot{u}_0=(\dot{\lambda}_0I-\dot{\Delta}_0^D)u_0&\text{ in }\Omega\\
    \end{array}
\end{equation}
with boundary conditions
\begin{align}
\label{eq:LR-nonweak-form-Neu}
  (\dot{A}_0\nabla u_0 + A_0\nabla \dot{u}_0)\bullet \nu & = 0 \qquad \text{in the Neumann, resp.}\\
\label{eq:LR-nonweak-form-Dir}
   \dot{u}_0 & = 0 \qquad \text{in the Dirichlet case.}
\end{align}
In order to see this, multiply \eqref{eq:LR-nonweak-form} with a test function $\varphi$ and apply the divergence theorem,this yields
\begin{equation*}
\begin{aligned}
-\int_\Omega  A_0\nabla \dot{u}_0 \bullet \nabla\varphi \; d\ell +& \int_{\partial\Omega}\varphi\cdot A_0\nabla \dot{u}_0\bullet\nu \;d\ell_{n-1} -
\lambda_0\int_\Omega   \dot{u}_0\cdot \varphi \; d\ell\\
&=
\dot{\lambda}_0 \int_\Omega  u_0\cdot \varphi \; d\ell +
\int_\Omega \dot{A}_0\nabla u_0 \bullet \nabla \varphi \; d\ell - \int_{\partial\Omega}\varphi\cdot \dot{A}_0\nabla u_0\bullet\nu \;d\ell_{n-1}.
\end{aligned}
\end{equation*}
The boundary integrals either vanish if $\varphi\in H^1_0(\Omega)$ (the Dirichlet case) or if the (natural) boundary condition \eqref{eq:LR-nonweak-form-Neu} is satisfied.

\begin{remark}
We note that the expression (\ref{eq:LR-nonweak-form}) is reminiscent of the classical linear response formula for the invariant density of a deterministic dynamical system.
In this setting, one has a family of transfer operators $\{\mathcal{L}_\eps\}$ generated by a family of maps $\{T_\eps\}$.
The (typically assumed unique) fixed point $h_\eps$ of $\mathcal{L}_\eps$ is the invariant density of $T_\eps$.
It is easy to verify the identity $(I-\mathcal{L}_\eps)(h_\eps-h_0)=(\mathcal{L}_\eps-\mathcal{L}_0)h_0$.
Dividing through by $\eps$ and taking the limit as $\eps\to 0$, one is able to show in certain situations that the limits $\dot{h}_0:=\lim_{\eps\to 0}(h_\eps-h_0)/\eps$ and $\dot{\mathcal{L}}_0:=\lim_{\eps\to 0}(\mathcal{L}_\eps-\mathcal{L}_0)/\eps$ exist in suitable senses, see e.g.\ \cite{liverani_notes}.
This leads to $(I-\mathcal{L}_0)\dot{h}_0=\dot{\mathcal{L}}_0h_0$, which is of the form (\ref{eq:LR-nonweak-form}) with $h_0, \dot{h}_0, \mathcal{L}_0, \dot{\mathcal{L}}_0$ replaced by $u_0, \dot{u}_0, \Delta^D_0, \dot{\Delta}^D_0$, respectively, noting that $\lambda_0=1$ and $\dot{\lambda}_0=0$.
\end{remark}

\section{Computing the linear response numerically}
\label{section-numerical-computation}

We now describe how to compute the linear response $\dot u_0$ numerically.
To this end, we approximately solve the weak form \eqref{eq:LR_weak} using the method described in \cite{FrJu18}.  That is, we consider \eqref{eq:LR_weak} on a finite-dimensional approximation space $V_N\subset H$, denoting the approximations of $\lambda, \dot{\lambda}, u_0, \dot{u_0}$ by $\tilde{\lambda}, \dot{\tilde{\lambda}}, \tilde{u}_0, \dot{\tilde{u}}_0$, respectively.
Instead of choosing $V_N$ as a subspace of $V_0$ (as would be required by \eqref{eq:LR_weak}), we enforce $\dot{\tilde u}_0\in \tilde V_0:=\spn(\tilde u_0)^\bot$ by adding an additional constraint.
In practice, the approximation space will be realised as a finite element space, typically using linear triangular Lagrange elements.

In \cite{FrJu18}, two different variants of a finite-element discretisation of the basic
eigenproblem for the dynamic Laplacian have been proposed, one based on the evaluation
of the right Cauchy Green deformation tensor (the \textit{CG method}) and one based on an
explicit approximation of the transfer operator associated to $T_\eps$ (the \textit{TO method}).
We now describe how to use both variants in order to compute $\dot u_0$.

\subsection{The CG Method}

 Let $\varphi_1,\ldots,\varphi_N$ be a basis for $V_N$.  As described in \cite{FrJu18} we obtain an approximation $\tilde\lambda_0,\tilde u_0$ of the eigenpair $\lambda_0,u_0$ by solving the matrix eigenproblem
\[
K \tilde {\bm u}_0 = \tilde\lambda_0 M\tilde {\bm u}_0,
\]
where
\begin{equation}\label{eq:KandM}
\begin{aligned}
K = -\left(\int_\Omega  A_0\nabla \varphi_j \bullet \nabla\varphi_k \;d\ell\right)_{j,k}, \quad
M = &\left(\int_\Omega  \varphi_j\cdot \varphi_k \; d\ell\right)_{j,k}
\end{aligned}
\end{equation}
are the \emph{stiffness} and \emph{mass} matrix, respectively, and ${\tilde{\bm{u}}}_0\in \R^N$ is the vector of coefficents of $\tilde u_0$ with respect to the chosen basis.

Similarly, we define the Galerkin approximation $\dot{\tilde{u}}_0 \in V_N$ of $\dot{u}_0$  by requiring it to satisfy \eqref{eq:LR_weak} for $\varphi=\varphi_j, j=1,\ldots,N$. This yields the linear system
\begin{equation}\label{eq:Galerkin_Matrix}
\begin{aligned}
(K  - \tilde\lambda_0 M) \dot{\tilde{\bm{u}}}_0 = (\dot{\tilde \lambda}_0 M - L) {\tilde{\bm{u}}}_0,
\end{aligned}
\end{equation}
for the coefficient vector $\dot{\tilde{\bm{u}}}_0$ of $\tilde u_0$ with respect to the basis $\varphi_1,\ldots,\varphi_N$. Here,
\begin{equation}\label{Mdef}
\begin{aligned}
L = -\left(\int_\Omega  \dot{A}_0 \nabla \varphi_j \bullet \nabla\varphi_k \;d\ell\right)_{j,k}
\end{aligned}
\end{equation}
is the ``linear response'' matrix.

Instead of choosing $V_N$ as a subspace of $V_0$, we enforce $\dot{\tilde u}_0\in \tilde V_0:=\spn(\tilde u_0)^\bot$ by adding an additional constraint on the coefficient vectors ${\tilde{\bm{u}}}_0$ and $\dot{\tilde{\bm{u}}}_0$, namely
\begin{equation}
\label{eq:orthogonality}
{\tilde{\bm{u}}}_0^\top M \dot{\tilde{\bm{u}}}_0 = 0
\end{equation}
which we append to \eqref{eq:Galerkin_Matrix}. We combine \eqref{eq:Galerkin_Matrix} and \eqref{eq:orthogonality} into a single linear system which allows to solve for $\dot{\tilde u}_0$ and $\dot{\tilde \lambda}_0$ simultaneously:
\begin{equation}
\label{eq:extended_system}
\left[
    \begin{array}{cc}
      K-\tilde\lambda_0 M & -M\tilde{\bm{u}}_0 \\
      {\tilde{\bm{u}}}_0^\top M & 0\\
    \end{array}
  \right]\left[
               \begin{array}{c}
                \dot{\tilde{\bm{u}}}_0 \\
                \dot{\tilde\lambda}_0 \\
                \end{array}
          \right]
 = \left[\begin{array}{c}
 -L\tilde{\bm{u}}_0 \\
 0\\
 \end{array}
 \right]
 \end{equation}

Note that according to our standing assumption, $\lambda_0$ is simple and so $\tilde\lambda_0$ is simple if the elements are fine enough (cf.~\cite{ErnGuermond04}, Lemma 3.65 and  \cite{schilling2019higher}) and the kernel of the matrix $K-\tilde\lambda_0 M$ is spanned by $\tilde {\bm u}_0$. Thus, on $\tilde V_0$, the matrix $K-\tilde\lambda_0 M$ is nonsingular and ($K, M$ and $L$ are symmetric)
\begin{align*}
\tilde{\bm u}_0^\top (K-\tilde\lambda_0 M) \dot{\tilde{\bm u}}_0 =
\tilde{\bm u}_0^\top (\dot{\tilde\lambda}_0 M-L) \tilde{\bm u}_0 = 0,
\end{align*}
i.e.\ the right hand side $(\dot{\tilde\lambda}_0 M-L) \tilde{\bm u}_0$ of \eqref{eq:Galerkin_Matrix} is in $\tilde V_0$.  The system \eqref{eq:Galerkin_Matrix} therefore has a unique solution on $\tilde V_0$ or equivalently:
 \begin{prop}
The linear system \eqref{eq:extended_system} has a unique solution.
 \end{prop}

\subsection{The TO Method}

The second variant of the finite-element based computation of the linear response $\dot{\tilde u}$ employs an explicit approximation of the transfer operator $T_{0,*}$ associated to $T_0$.  It yields an alternative way to compute the matrices $K$ and $L$ in \eqref{Mdef} -- everything else remains unchanged from the previous section.  In particular, we again solve the linear system \eqref{eq:extended_system} in order to obtain the approximate linear response $\dot{\tilde{u}}_0$.

\paragraph{Approximating the transfer operator.} In addition to $V_N$, we choose a finite-dimensional subspace $V_N^1\subset H_0^1(T_0(\Omega))$ in the case of Dirichlet boundary conditions (resp.\ $V_N^1\subset H^1(T_0(\Omega))$ in the case of Neumann boundary conditions, cf.\ the discussion on the appropriate spaces in the preceeding section). Let $\varphi_1^1,\ldots,\varphi_N^1$ be a basis of $V_N^1$. In order to approximate
\[
T_{0,*}\varphi_j \approx \sum_{k=1}^N \alpha_{jk} \varphi_k^1,
\]
we choose a set $\{x_1^1,\ldots,x_N^1\}$ of sample points in $T_0(\Omega)$ and require that
\[
T_{0,*}\varphi_j(x_m^1) = \sum_{k=1}^N \alpha_{jk} \varphi_k^1(x_m^1)
\]
for $j,m=1,\ldots,N$. If the $\varphi_j^1$ are a nodal basis with respect to the sample points $x_1^1,\ldots,x_N^1$, then $\varphi_k^1(x_m^1)=\delta_{km}$ and thus
$
\alpha_{jm} = T_{0,*}\varphi_j(x_m^1) = \varphi_j(T_{0,*}^{-1}(x_m^1)).
$
In particular, if the sample points $x_m^1$ are chosen as the image points $x_m^1=T_0(x_m)$ of the sample points $x_1,\ldots,x_N$ in $\Omega$ and the basis $\varphi_1,\ldots,\varphi_N$ is a nodal basis with respect to the points $x_1,\ldots,x_N$, then $\alpha_{jm}=\varphi_j(T_{0,*}^{-1}(x_m^1)) = \varphi_j(x_m)=\delta_{jm}$, i.e.\ the representation matrix ${\bm\alpha}:=(\alpha_{jm})_{jm}$ of $T_{0,*}$ with respect to the two nodal bases $\varphi_1,\ldots,\varphi_N$ and $\varphi_1^1,\ldots,\varphi_N^1$ is the identity matrix.
This latter case is the ``adaptive'' TO method from \cite{FrJu18}, where here we are considering only two discrete time-instances.

\paragraph{Approximating the stiffness matrix $K$.} With
$
K^0 := \left(\int_{\Omega}  \nabla  \varphi_j\bullet\nabla \varphi_k \; d\ell\right)_{j,k}
$
and $K^1:=\bm\alpha^\top K^0\bm\alpha$ we obtain
\begin{equation}
K=-\frac12(K^0+K^1) \approx -\frac{1}{2}\left(\int_{\Omega}  \nabla  \varphi_j\bullet\nabla \varphi_\ell \; d\ell + \int_{T_0(\Omega)}  \nabla (T_{\eps,*} \varphi_j)\bullet\nabla (T_{0,*}\varphi_\ell) \; d\ell\right)
\end{equation}
as an approximation to the stiffness matrix in \eqref{eq:KandM}.

\paragraph{Approximating the stiffness response matrix $L$.} We next describe an alternative way to compute the matrix $L$ in \eqref{Mdef} based on the explicit approximation of the transfer operater described above. We will use only function evaluations of $T_0$ and $\dot{T}_0$. First, we manipulate the expression for $L$.
\begin{prop}\label{prop:dotAfromT*}
For $f,g\in H^1(\Omega)$,
\begin{equation*}\label{eq:analytic-form-L-TO}
\begin{aligned}
-\int_\Omega\dot{A}_0\nabla f\bullet\nabla g\ d\ell
&= \int_{T_0(\Omega)}   D(T_{0,*}\dot{T}_0)^{sym}\nabla T_{0,*}f \bullet \nabla T_{0,*}g\;d\ell.
\end{aligned}
\end{equation*}
\end{prop}

\begin{proof}
We recall from \eqref{eq:derv-mat-coef} that
$\dot{A}_0 = -\left((DT_0)^{-1}D\dot{T}_0 (DT_0)^{-1}(DT_0)^{-\top}\right)^\sym$.
Next, we compute
\begin{align*}
    & \int_\Omega  (DT_0)^{-1}D\dot{T}_0 (DT_0)^{-1}(DT_0)^{-\top}\nabla f \bullet \nabla g \;d\ell \\
&= \int_\Omega  D\dot{T}_0 (DT_0)^{-1}(DT_0)^{-\top}\nabla f \bullet (DT_0)^{-\top}\nabla g \;d\ell \\
&= \int_\Omega  (D\dot{T}_0) (DT_0)^{-1}(\nabla T_{0,*}f\circ T_0) \bullet (\nabla T_{0,*}g\circ T_0) \;d\ell\\
&= \int_{T_0(\Omega)}  (D\dot{T}_0\circ T_0^{-1})\ ((DT_0)^{-1}\circ T_0^{-1})\nabla T_{0,*}f \bullet \nabla T_{0,*}g \;d\ell\\
&= \int_{T_0(\Omega)}  (D\dot{T}_0\circ T_0^{-1})\ DT^{-1}_0\nabla T_{0,*}f \bullet \nabla T_{0,*}g \;d\ell\\
&= \int_{T_0(\Omega)}  D(T_{0,*}\dot{T}_0)\nabla T_{0,*}f \bullet \nabla T_{0,*}g \;d\ell\\
\end{align*}
Similarly, we have
\begin{align*}
    \int_\Omega  \left((DT_0)^{-1}D\dot{T}_0 (DT_0)^{-1}(DT_0)^{-\top}\right)^\top\nabla f \bullet \nabla g \;d\ell
     = \int_{T_0(\Omega)}   (D(T_{0,*}\dot{T}_0))^\top \nabla T_{0,*}f \bullet \nabla T_{0,*}g\;d\ell.
\end{align*}
We thus obtain
\begin{equation*}\begin{aligned}
-\int_\Omega\dot{A}_0\nabla f\bullet\nabla g\;d\ell
&= \int_{T_0(\Omega)}   D(T_{0,*}\dot{T}_0)^{sym} \nabla T_{0,*}f \bullet \nabla T_{0,*}g\;d\ell.
\end{aligned}
\end{equation*}
\end{proof}

The right-hand-side of the expression in Proposition \ref{prop:dotAfromT*} has two types of terms, we now discuss their approximation.  Given a function $g\in H^1(\Omega)$ we approximate it in $V_N$ by
\[
g \approx \sum_{j}g_j\varphi_j,
\]
where $g_j=g(x_j)$, since we assume $\varphi_1,\ldots,\varphi_N$ to be a nodal basis on the nodes $x_1,\ldots,x_N$.  In the case of the adaptive TO method, the approximation of the pushforward $T_{0,*}g$ is therefore given by $\sum_{j,k} g_j\alpha_{jk}\varphi_k^1=\sum_k g_k\varphi_k^1$ (because $\alpha_{jk}=\delta_{jk})$.
Finally, following \cite{FrJu18} we approximate
\[
\nabla T_{0,*} g \approx \sum_{k} g_k\nabla\varphi_k^1,
\]
in particular, we have $\nabla T_{0,*} \varphi_j \approx \nabla\varphi_j^1$.
Now we discuss the approximation of the term $D(T_{0,*}\dot{T}_0)$ in \eqref{eq:analytic-form-L-TO}.
We denote by $\dot{T}_{0,1},\dots, \dot{T}_{0,n}$ the component functions of $\dot{T}_0$. Correspondingly,
\[
T_{0,*}\dot{T}_0 =\dot{T}_0\circ T_0^{-1} = ( \dot{T}_{0,1}\circ T_0^{-1},\dots, \dot{T}_{0,n}\circ T_0^{-1}) =( T_{0,*}\dot{T}_{0,1},\dots, T_{0,*}\dot{T}_{0,n}).
\]
Since each $\dot{T}_{0,i}$ is a scalar-valued function, we approximate $\nabla (T_{0,*}\dot{T}_{0,i})$ in exactly the same way as we approximated $\nabla T_{0,*} g$, namely, we write
$\nabla (T_{0,*}\dot{T}_{0,i}) \approx \sum_{s} w_s^i\nabla\varphi^1_s$, where $w_s^i=\dot{T}_{0,i}(x_s)$.
Thus
\begin{align*}
    D(T_{0,*}\dot{T}_0) = \left[\begin{array}{c}
(\nabla T_{0,*}\dot{T}_{0,1})^\top \\
\vdots \\
(\nabla T_{0,*}\dot{T}_{0,n})^\top
\end{array} \right] \approx  W\; D\Phi,
\end{align*}
where $W=(w_s^k)_{ks}$ and $D\Phi=(\partial_k\varphi_s^1)_{sk}$.
We then obtain the approximation
\begin{align*}
      D(T_{0,*}\dot{T}_0) \nabla T_{0,*} \varphi_j \bullet \nabla T_{0,*} \varphi_k
& \approx  W D\Phi \;\nabla\varphi_j^1 \bullet \nabla\varphi_k^1
\end{align*}
which, using Proposition~\ref{prop:dotAfromT*}, yields
\begin{equation*}
\begin{aligned}
L_{jk} = -\int_\Omega \dot{A}_0\nabla \varphi_j\bullet\nabla \varphi_k \;d\ell
& \approx  \int_{T_0(\Omega)} (W D\Phi)^\text{sym} \;\nabla\varphi_j^1 \bullet \nabla\varphi_k^1\;d\ell
\end{aligned}
\end{equation*}
as approximations for the entries of the matrix $L$ in \eqref{Mdef}.


\section{Experiments}\label{sec:experiments}

The code for the following experiments is available in the \href{https://github.com/gaioguy/FEMDL}{\texttt{FEMDL}}\footnote{Available at \href{https://github.com/gaioguy/FEMDL}{\texttt{https://github.com/gaioguy/FEMDL}}} package. Since we identify coherent sets as level sets of eigenfunctions, and are interested in the evolution of coherent sets, we will begin this section with a short note about the evolution of level sets.

\subsection{Level-Set Evolution}\label{sec:lvl-set-evo}

We wish to describe the change of the level sets of $u_\eps$ as we perturb the parameter $\eps$. From the level-set method \cite{osher1988fronts}, we note the following. For $\eps\in[-\eps_0,\eps_0]$, let $\Gamma_\eps = \{x\in \Omega: u_\varepsilon(x)=c\}$ be a differentiable family of nontrivial closed curves in $\Omega$.
Because $u_\varepsilon$ varies as $\eps$ increases from 0, so too do the curves $\Gamma_\eps$.
Define the function $s:\Gamma_0\to\mathbb{R}$ to be the instantaneous speed of the curve $\Gamma_0$ with respect to $\eps$ in the direction normal to $\Gamma_0$.
Then $s$ satisfies the level-set equation
\begin{equation}
\dot{u}_0 + s |\nabla u_0| = 0
\end{equation}
and
$$\frac{\partial\Gamma_\eps}{\partial\eps}|_{\eps=0} = \frac{-\dot{u}_0}{|\nabla u_0|}\nabla u_0.
$$
Extending this formula to all of $\Omega$ we obtain a vector field
$$v_{\rm level}:=\frac{-\dot{u}_0}{|\nabla u_0|}\nabla u_0,$$
which describes the instantaneous evolution of level sets of $u_0$.
We will use the vector field $v_{\rm level}$ in the following experiments to visualise the evolution of coherent sets with small changes in $\eps$.
We are not directly concerned with the possibility that the level sets occasionally undergo topological bifurcations as $\eps$ is varied, however, it is well known \cite{osher1988fronts} that such bifurcations are seamlessly captured by smooth evolution of the $u_\eps$ with $\eps$.

\subsection{The Standard Map}

We start with the standard map on the flat 2-torus, given by
\begin{equation}
\label{eq:stdmap}
T_\eps(x,y)=(x+y+(a+\eps)\sin x,y+(a+\eps)\sin x)\pmod{2\pi}, \quad a =0.98.
\end{equation}
The parameter $a+\eps$ controls the nonlinearity of the map and we investigate how varying $\eps$ from $0$ affects coherent sets.  For the computations, we use a Delaunay triangulation on a regular grid of $100\times 100$ points on the 2-torus and Gauss quadrature of degree 2 in order to approximate the integrals in the CG approach.

In Figure \ref{fig:stdmap_u0} (left), we show the eigenvector $u_0$ at the second eigenvalue $\lambda_0=-1.08$\footnote{Note that the first eigenfunction is constant.  The experiment works equally well for higher eigenfunctions, cf.\ the code at \href{https://github.com/gaioguy/FEMDL}{\texttt{https://github.com/gaioguy/FEMDL}}.} of the dynamic Laplacian $\Delta^D_0$ for the nominal value $a=0.98$ (corresponding to $\eps=0$), which identifies a coherent set in the center of the domain (red).
Figure \ref{fig:stdmap_u0} (second and third from left) displays $\dot{u}_0$ and $u_0+\eps \dot{u}_0$ for $\eps=0.5$, which -- even though only a linear extrapolation -- is quite similar to the exact eigenvector $u_\eps$ for $\eps=0.5$ at the second eigenvalue $\lambda_\eps = -1.23$ of $\Delta^D_\eps$ shown in Figure \ref{fig:stdmap_u0} (right). In fact, using again Gauss quadrature of order 2, we obtain the relative $L^2$-error $\|u_\eps - (u_0+\eps \dot{u}_0)\|/\|u_0\| \approx 0.03$. We also obtain $\dot\lambda_0 = -0.23$, which results in the estimate $\lambda_0+\eps\dot\lambda_0 = -1.19$  for $\lambda_\eps$, i.e.\ using $\dot\lambda_0$ we get an estimate of $\lambda_\eps$ with, again, a relative error of $3\%$.
Because the more negative $\lambda_0$ is, the less coherent the associated coherent sets, $\dot\lambda_0<0$ indicates a loss of coherence as $\eps$ is increased.
 These numbers and figures are obtained from the CG approach. The results from the TO approach are similar and in fact visually indistinguishable, so we do not show them here.  Note that this is an advantage for the TO approach since its computational effort is considerably lower and it only requires point evaluations of the flow map.
\begin{figure}[htbp]
  \centering
  \includegraphics[width=0.25\textwidth]{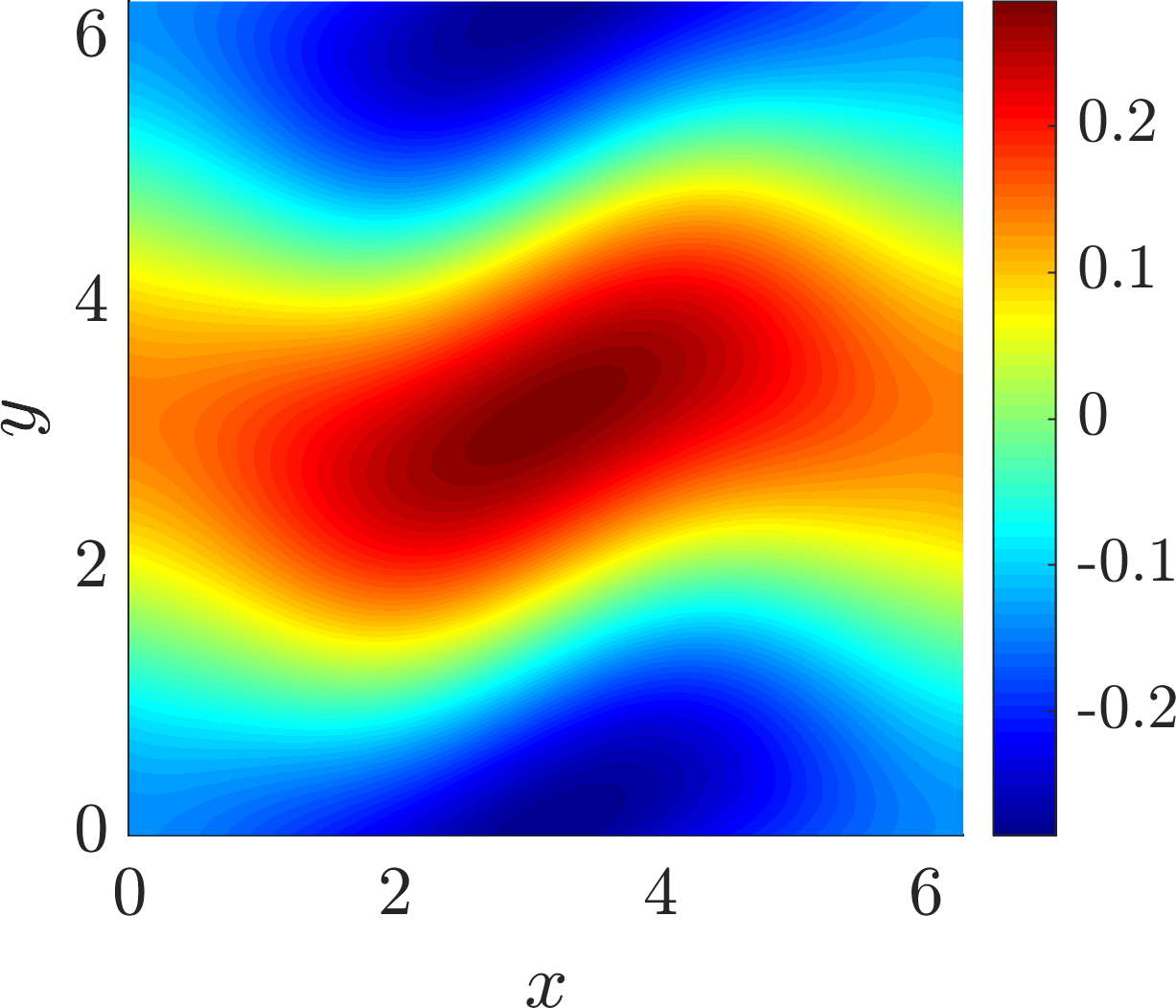}
  \includegraphics[width=0.24\textwidth]{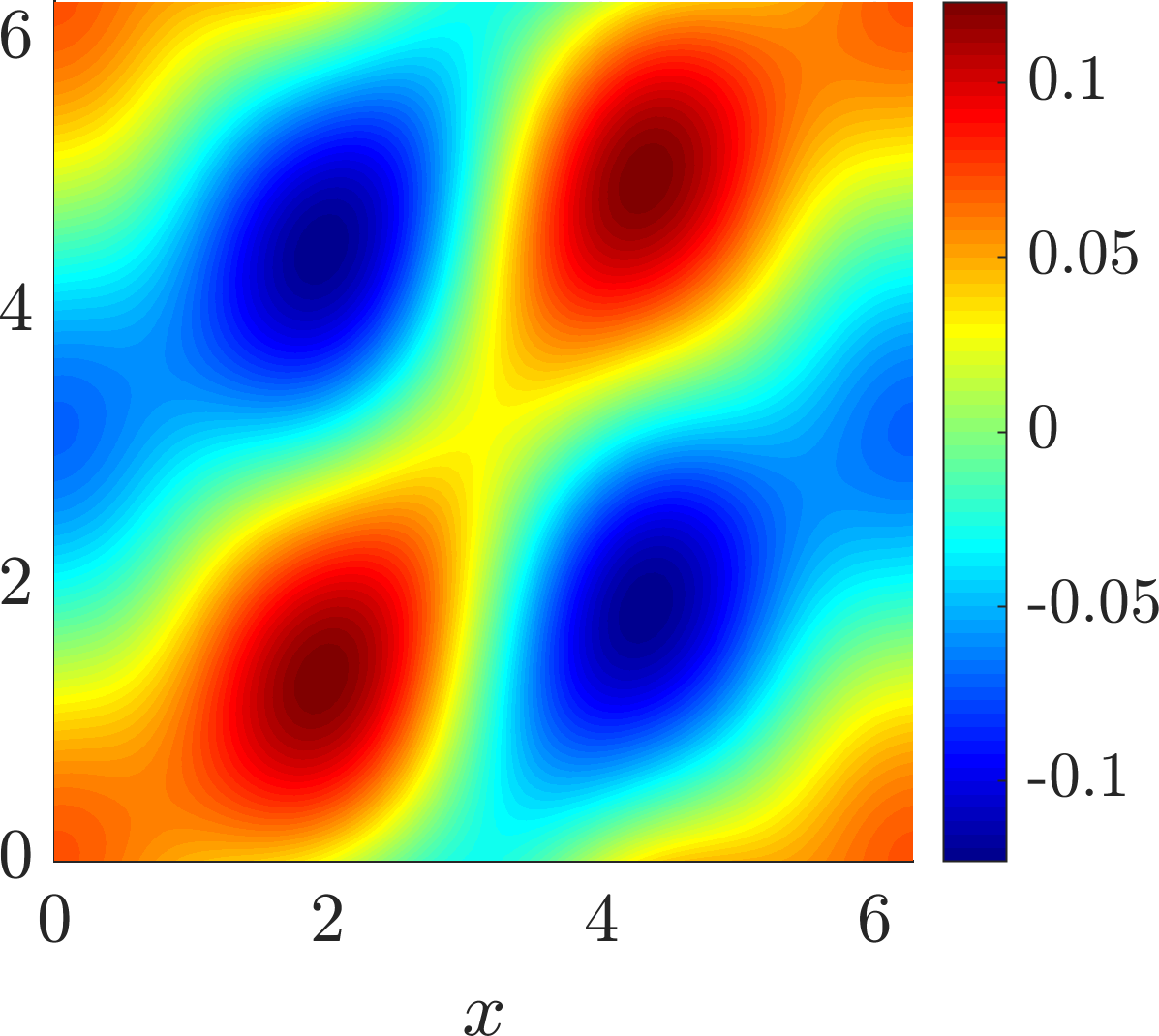}
  \includegraphics[width=0.24\textwidth]{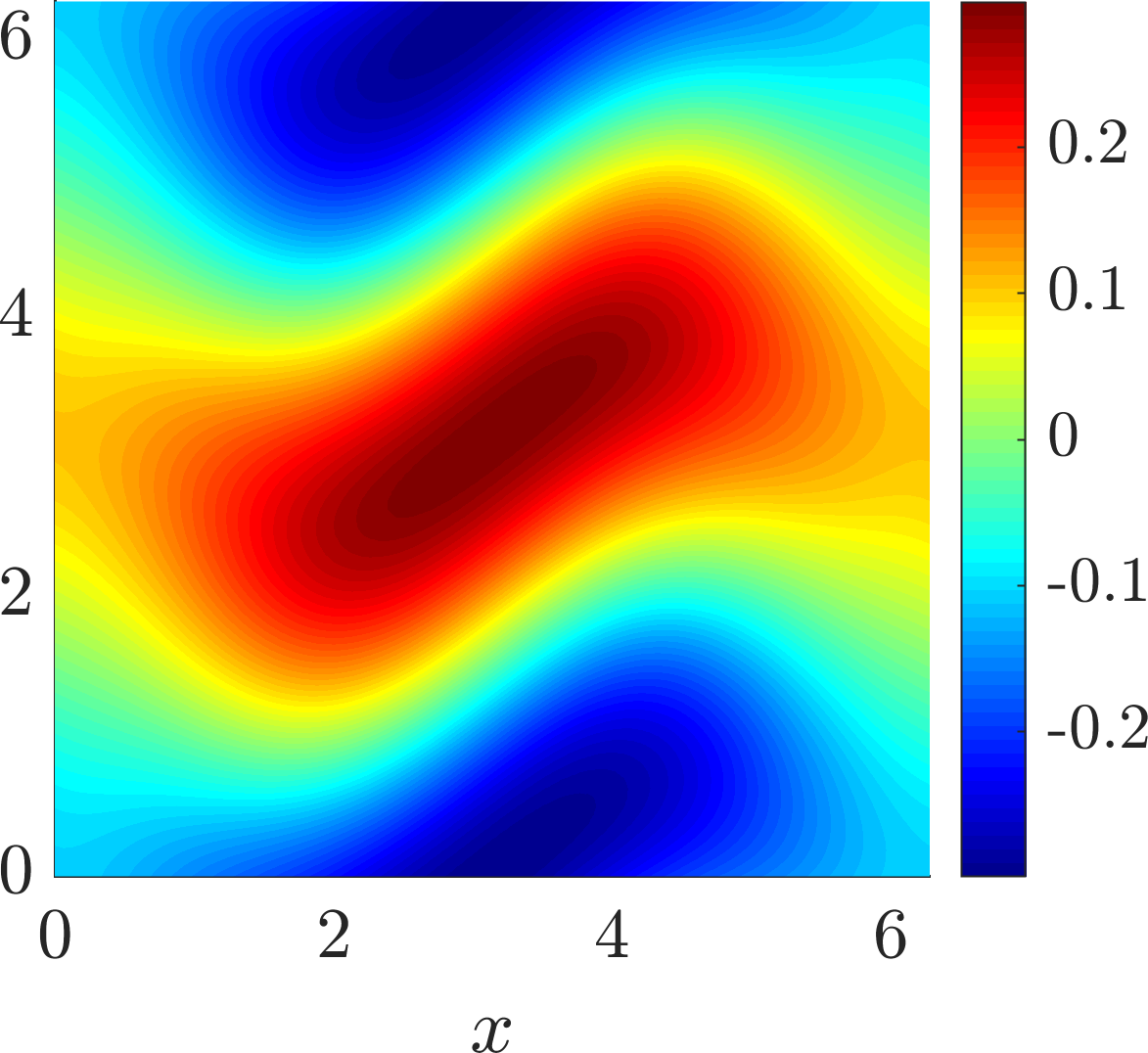}
  \includegraphics[width=0.24\textwidth]{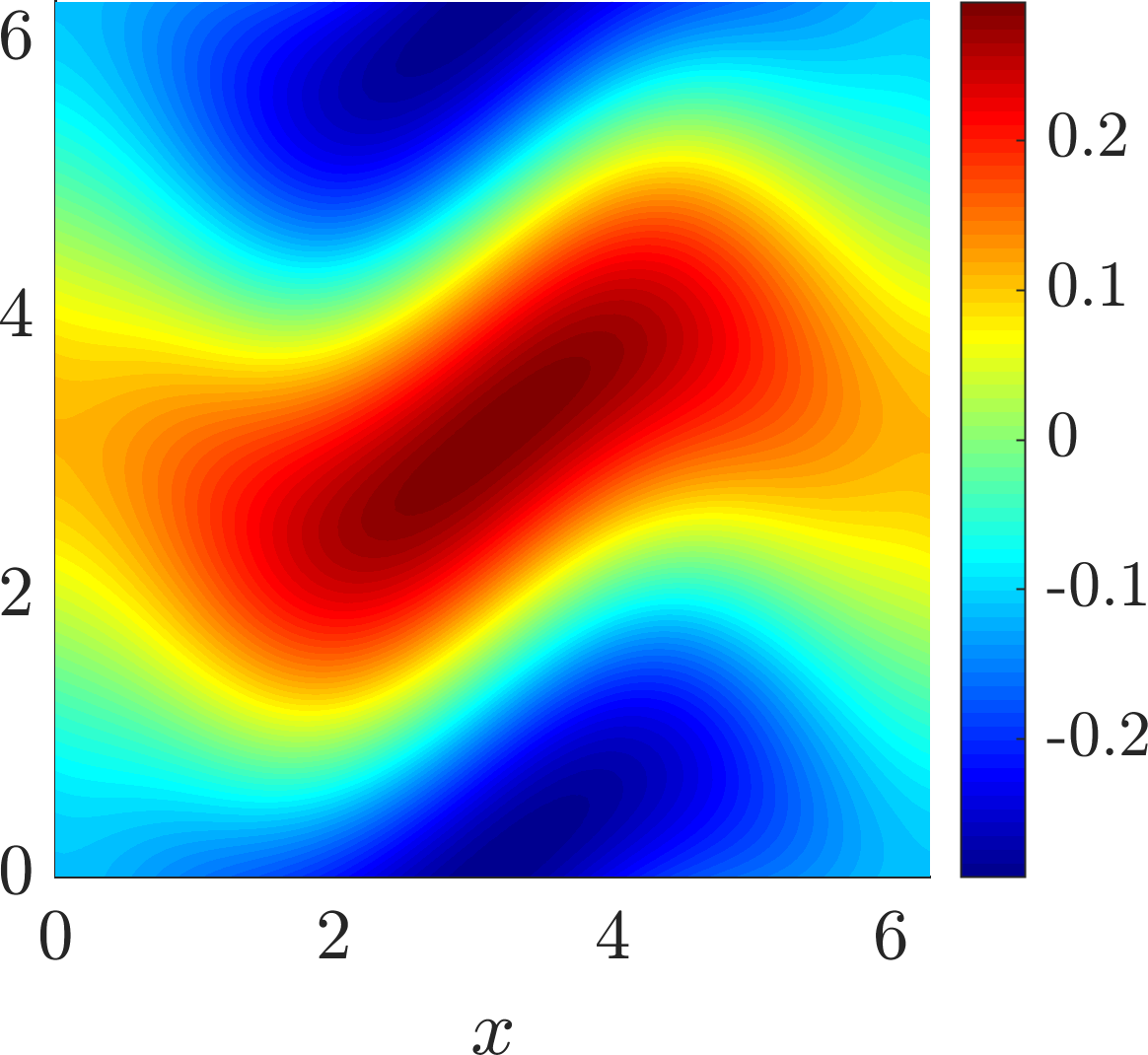}
  \caption{Standard map (left to right): $u_0$, $\dot{u}_0$, $u_0+\frac12\dot{u}_0$ and $u_\eps$, $a=0.98$, $\eps=0.5$.}
  \label{fig:stdmap_u0}
\end{figure}

Figure \ref{fig:stdmap_velocity} shows the velocity field for the level-set curves at $\eps=0$ which describes how the coherent set boundaries move in the fixed frame at $t_0$ as $a$ is varied from its nominal value $0.98$. We also show the level set at the value $c=0.1447$ which was selected from a line search of $c\in [0,\max_x u_0(x)]$ that minimises the dynamic Cheeger value in \eqref{cheeger} with $\Gamma_c=\{x\in \Omega: u_0(x)=c\}$.  As the parameter is increased from $a=0.98$ to a larger value, the boundary of the coherent set moves according to the velocity field $v_{\rm level}$ shown in Figure \ref{fig:stdmap_velocity}.  We further compare the level sets of $u_0$, $u_\eps$ and $u_0+\eps\dot u_0$, i.e. the prediction of $u_\eps$ by the linear Taylor approximation at $\eps=0.5$.  The predicted level set is indistinguishable from the true level set.  Note that we can obtain predictions for the perturbed level sets very cheaply by computing contours for $u_0+\eps \dot{u}_0$, $\eps\in [0,0.5]$.
\begin{figure}[htbp]
  \centering
  \includegraphics[width=0.25\textwidth]{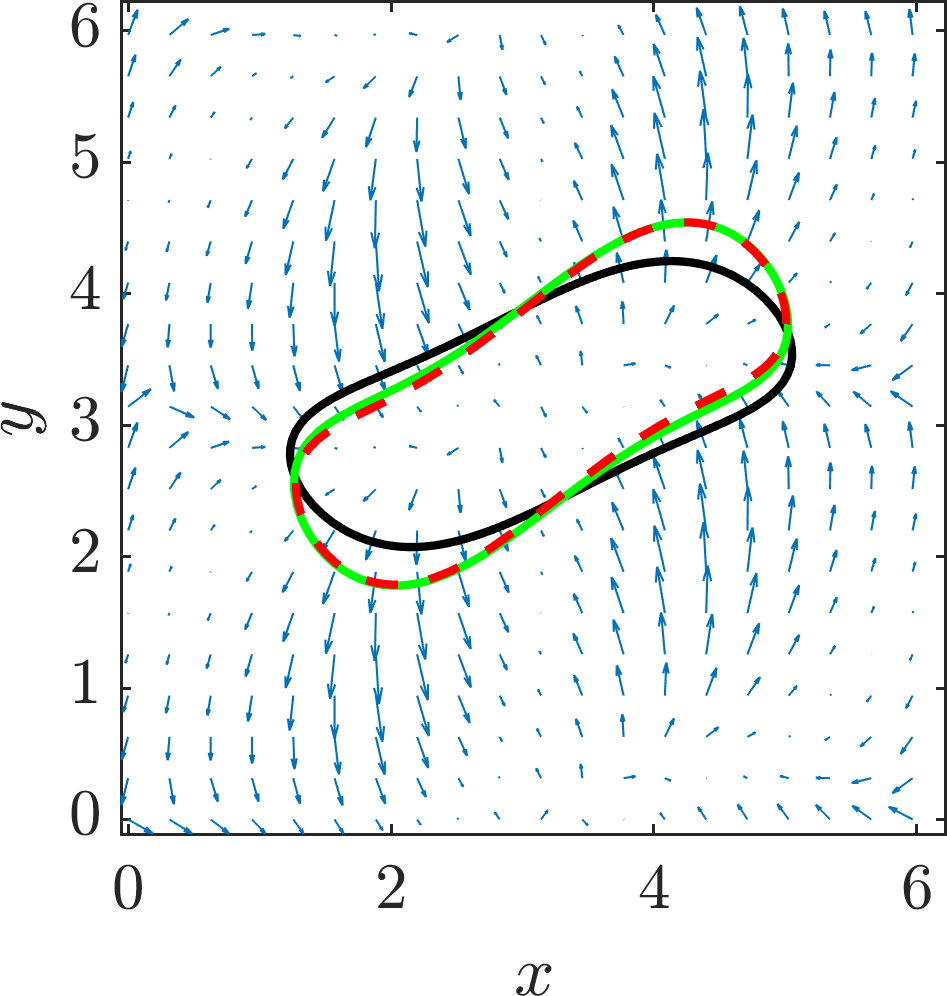}
  \caption{Standard map: velocity field $v_{\rm level}$ for the evolution of level sets (blue arrows); level sets of $u_\eps$ for $\eps=0$ (black) and $\eps=0.5$ (red dashed), and of $u_0+0.5\dot u_0$, i.e.\ the  prediction of $u_{0.5}$ by linear response (green).}
\label{fig:stdmap_velocity}
\end{figure}

\subsection{The transitory double gyre}
\label{exp:rotating_double_gyre}

In our second experiment, we consider the transitory flow from
\cite{mosovsky2011transport}.
This is a non-periodic time-variant Hamiltonian system with Hamiltonian $H=-\psi$, where $\psi$ is the stream function
\begin{align*}
\psi(x,y,t) &= (1-s(t))\psi_P(x,y) + s(t)\psi_F(x,y)\\
\psi_P(x,y) &= \sin(2\pi x)\sin(\pi y)\\
\psi_F(x,y) &= \sin(\pi x)\sin(2\pi y)
\end{align*}
and $s(t)$ is the transition function
\begin{align*}
  s(t) = \left\{\begin{array}{cl} 0 & \text{for } t<0,\\ t^2(3-2t) & \text{for }t\in [0,1],\\ 1 & \text{for }t>1.\end{array}\right.
\end{align*}
On the square $\Omega=[0,1]^2$, the vector field initially (at $t_0=0$) exhibits two gyres (if considered as a steady flow), with centers at $(\frac14,\frac12)$ and $(\frac34,\frac12)$. At the terminal time $t_1=1$, the vector field exhibits these gyres rotated by $\pi/2$ (again if considered as a steady flow).
In this experiment we view the flow time $t_1$ as the perturbation parameter, and analyse the effect on the coherent sets under the unsteady flow as the flow time is increased.

For the computations, we approximate the flow map and the solution of the variational equation by Matlab's \texttt{ode45} command, i.e.\ an explicit Runge Kutta-scheme with adaptive step size control. We construct a Delaunay triangulation of a regular $100\times 100$ grid of nodes on the square and use Gauss quadrature of degree 5 in order to approximate the integrals in the CG approach.

\begin{figure}[H]
  \centering
  \includegraphics[width=0.25\textwidth]{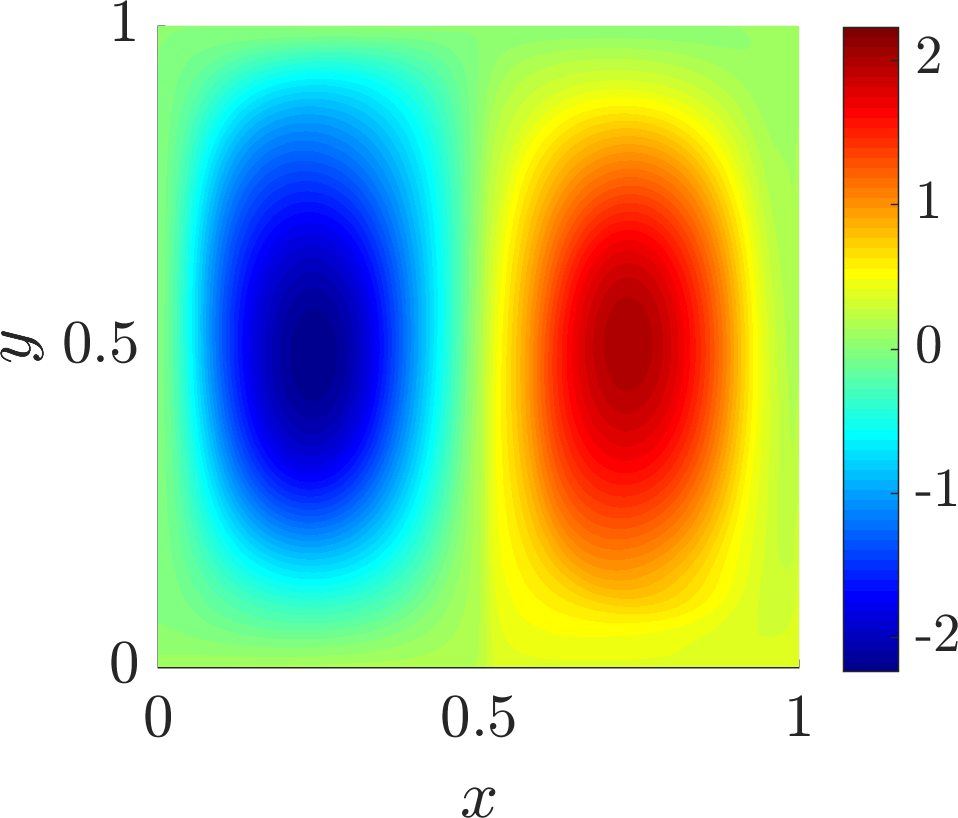}
  \includegraphics[width=0.23\textwidth]{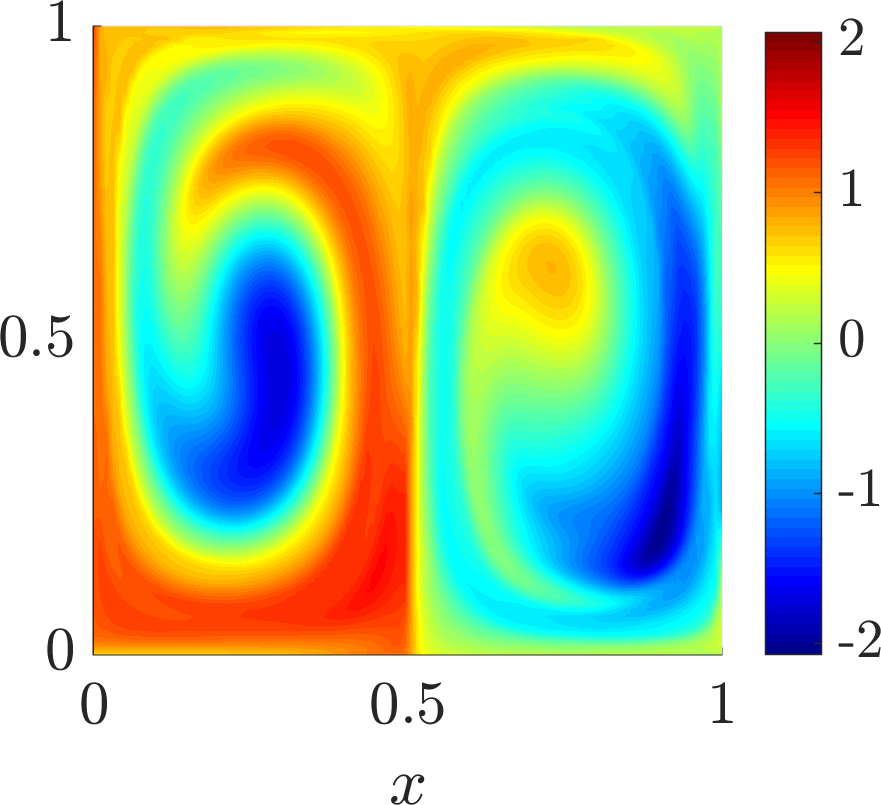}
  \includegraphics[width=0.23\textwidth]{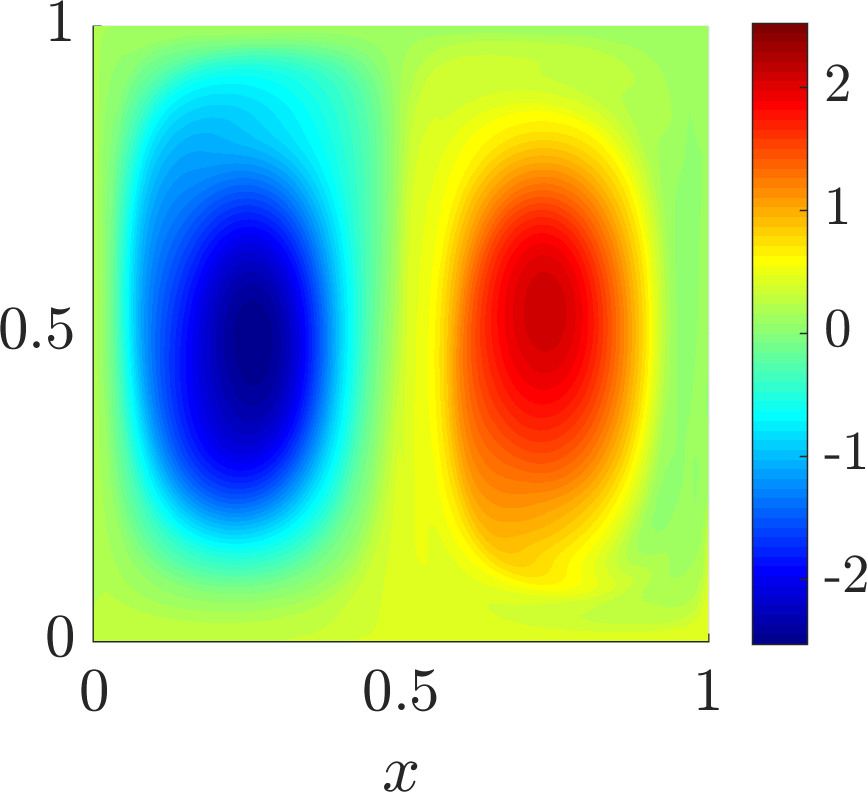}
  \includegraphics[width=0.23\textwidth]{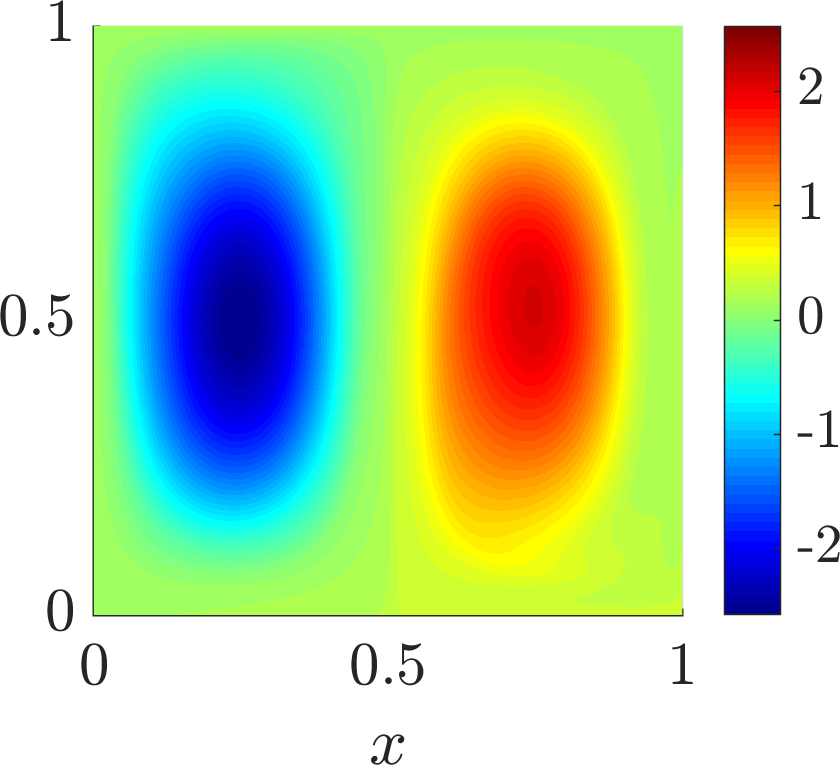}
  \includegraphics[width=0.25\textwidth]{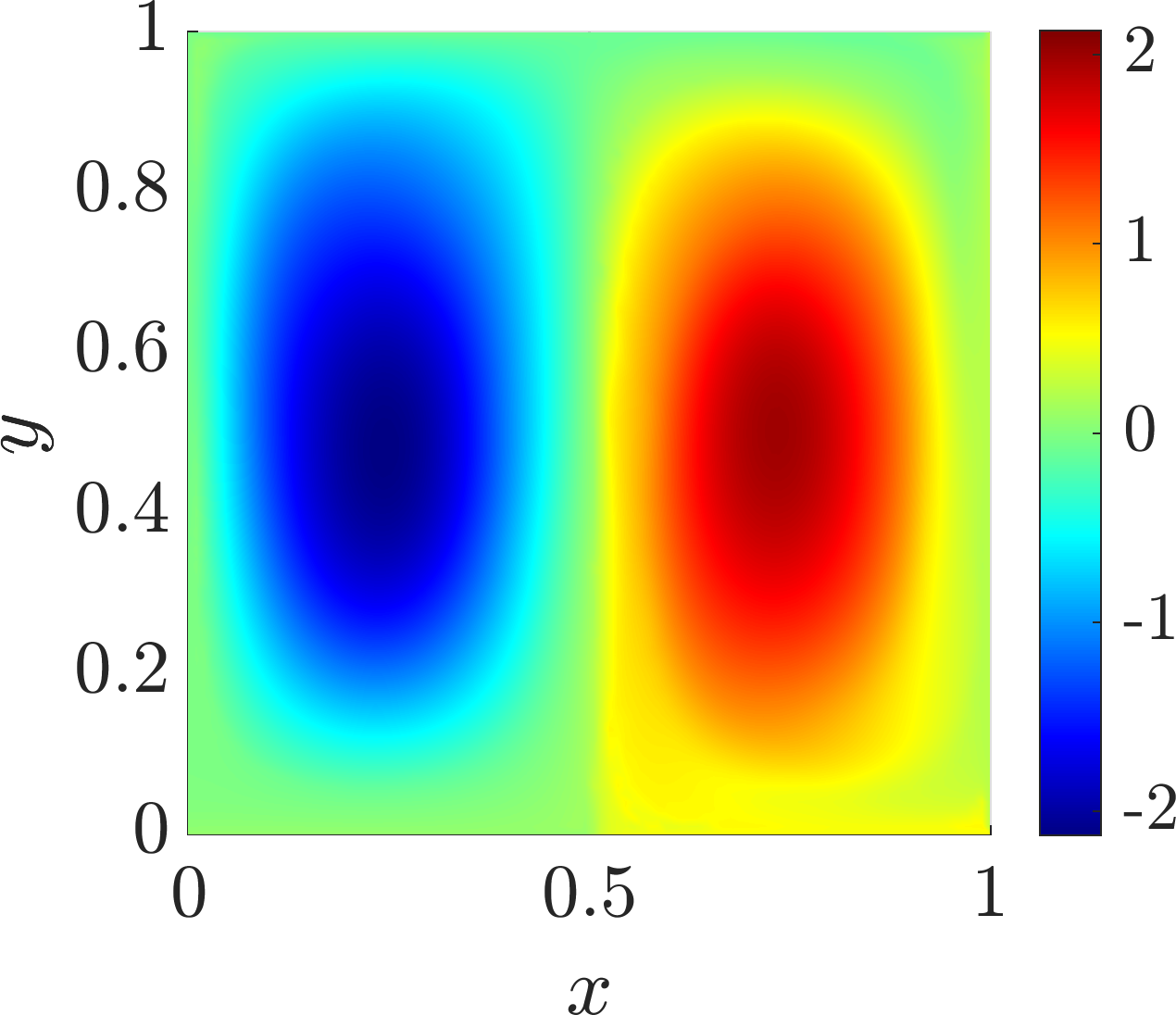}
  \includegraphics[width=0.23\textwidth]{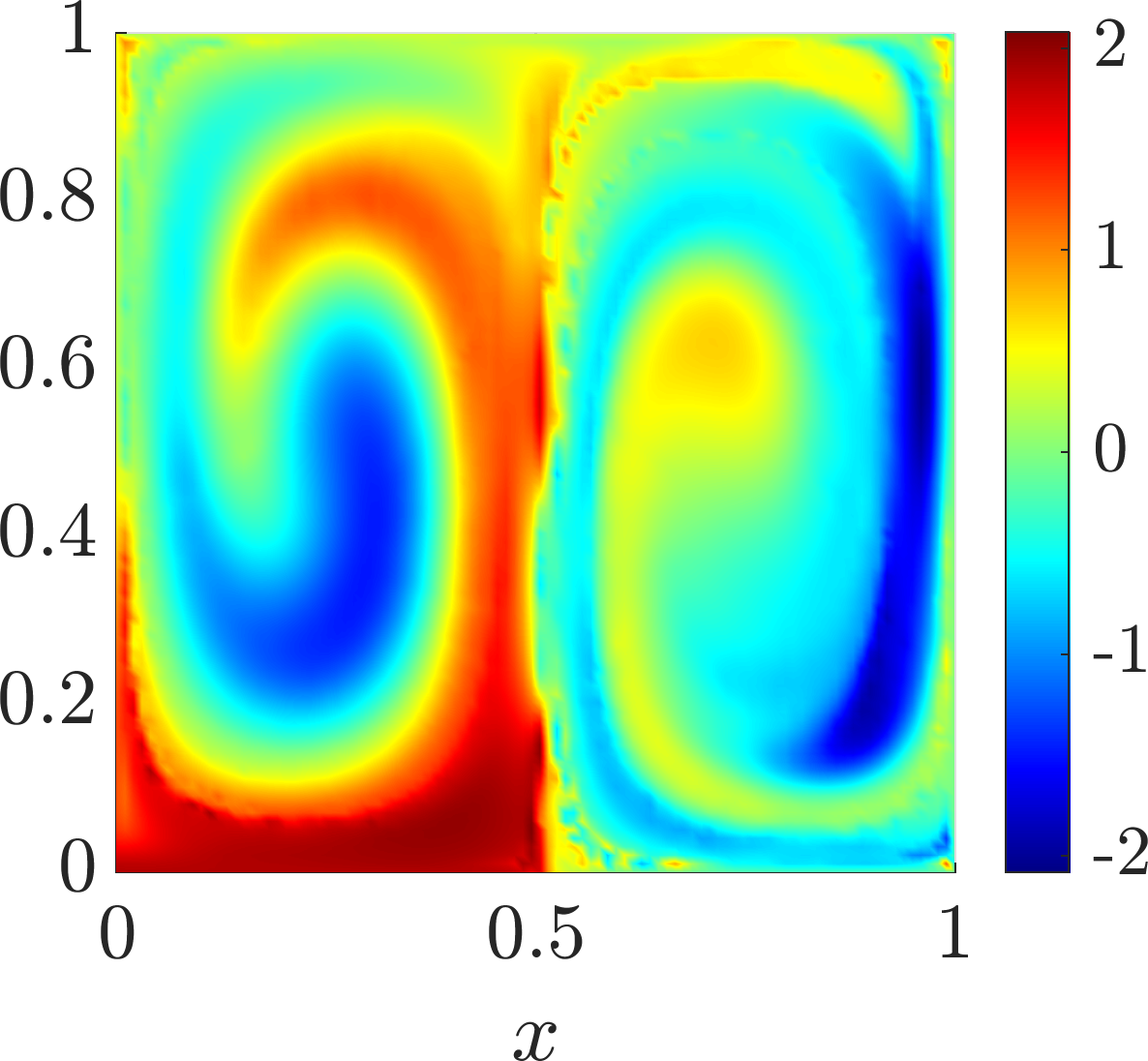}
  \includegraphics[width=0.23\textwidth]{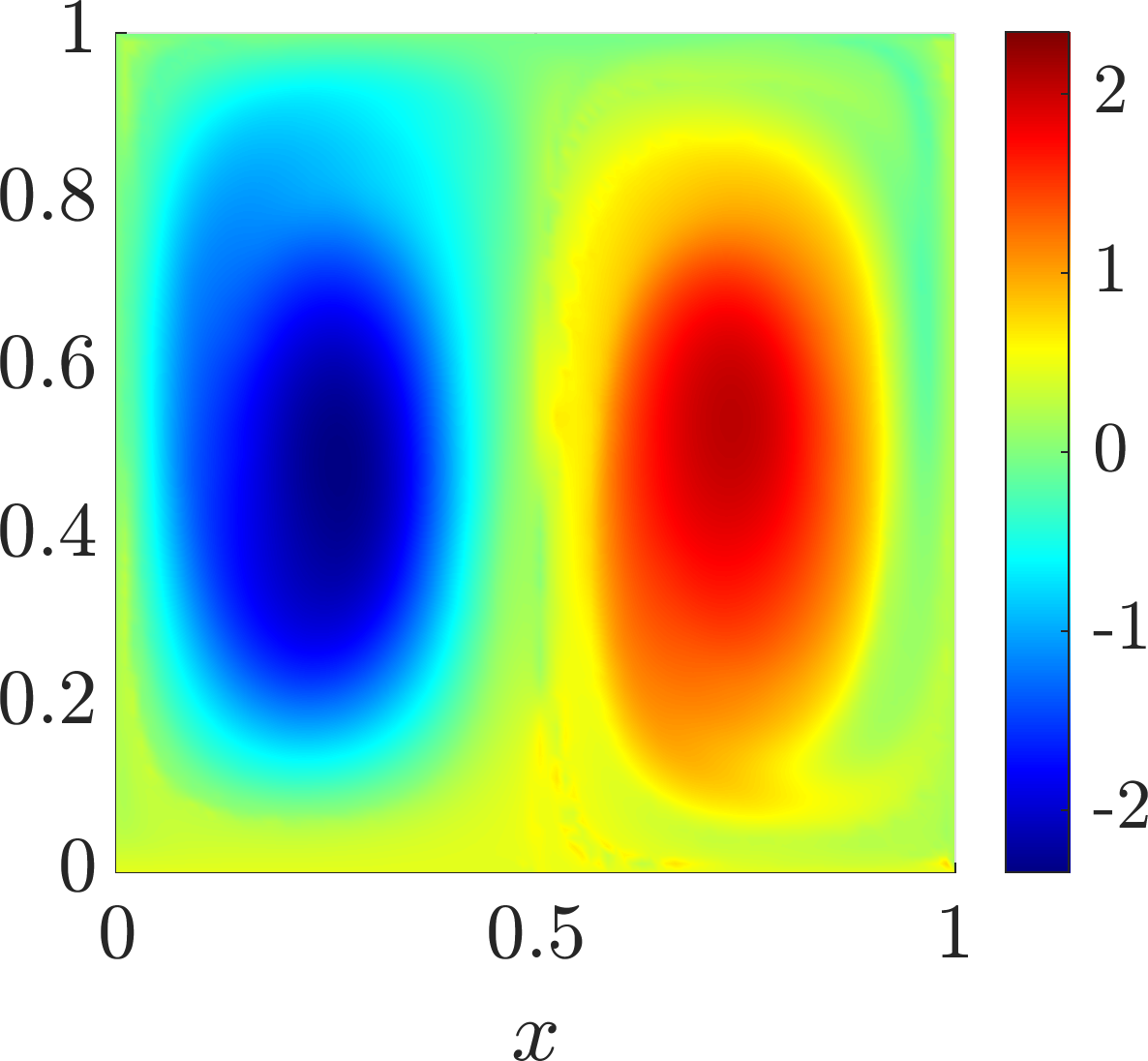}
  \includegraphics[width=0.23\textwidth]{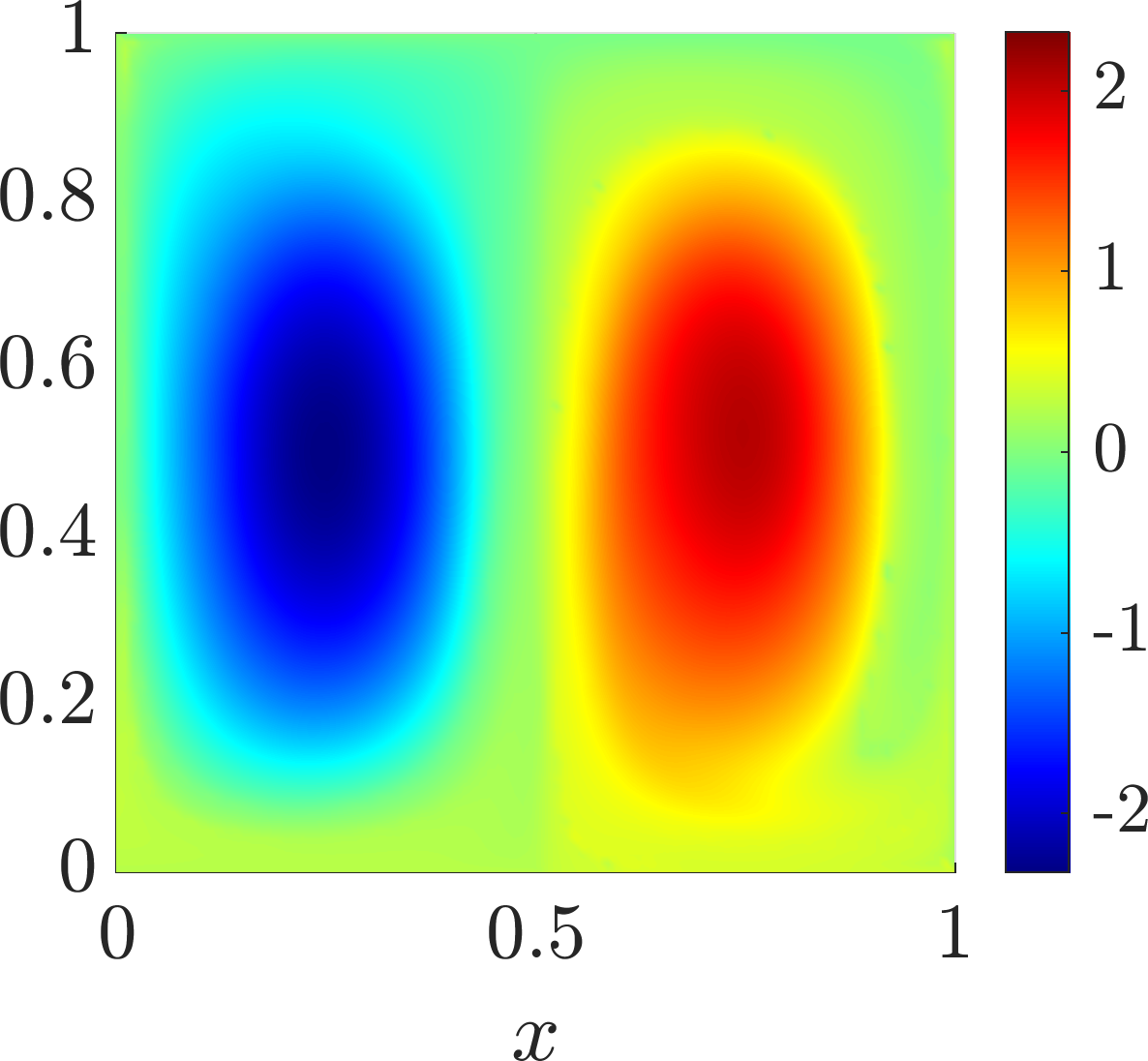}
  \caption{Transitory double gyre (left to right): $u_0$, $\dot{u}_0$, $u_0+0.2\dot{u}_0$ and $u_\eps$ (all plotted at time $t=0$). Top row: CG method, bottom row: TO method}
  \label{fig:rot_double_gyre_u0}
\end{figure}

In Figure \ref{fig:rot_double_gyre_u0} (left) we show the eigenvector $u_0$ at the second eigenvalue $\lambda_0=-50.4$ (all numbers from the TO approach) of the dynamic Laplacian $\Delta_0^D$ for $t_1=0.6$ (corresponding to $\eps=0$).
This figure is consistent with earlier experiments using transfer operator methods \cite[Figure 7(a)]{FrPa14} and Ulam approximation of the dynamic Laplacian \cite[Figure 8 (left)]{froyland2015dynamic}.
The eigenvector identifies two coherent sets displayed in red and blue.
The two plots in the center of Figure \ref{fig:rot_double_gyre_u0} display $\dot{u}_0$ (center left) and $u_0+\eps \dot{u}_0$ (center right) for $\eps=0.2$, corresponding to a flow time $t_1+\eps=0.6+0.2=0.8$.  Even for this quite large extrapolation value, we obtain a result that is very similar to the exact eigenvector $u_\eps$ at the second eigenvalue  $\lambda_\eps=-61.6$ shown in the very right column in Figure \ref{fig:rot_double_gyre_u0}. The corresponding relative $L^2$-error is $\|u_\eps-(u_0+\eps\dot u_0)\|/\|u_0\|\approx 0.1$. We further obtain  $\dot\lambda_0 = -50.4$, which results in the estimate $\lambda_0+\eps\dot\lambda_0 = -60.5$ for $\lambda_\eps$, i.e.\ using $\dot\lambda_0$ we get an estimate of $\lambda_\eps$ with a relative error of $2\%$.
As expected, lengthening the flow time leads to a loss of coherence, indicated by the negative sign of $\dot{\lambda_0}$, and approximately quantified by the magnitude of $\dot{\lambda_0}$.
\begin{figure}[H]
  \centering
  \includegraphics[width=0.3\textwidth]{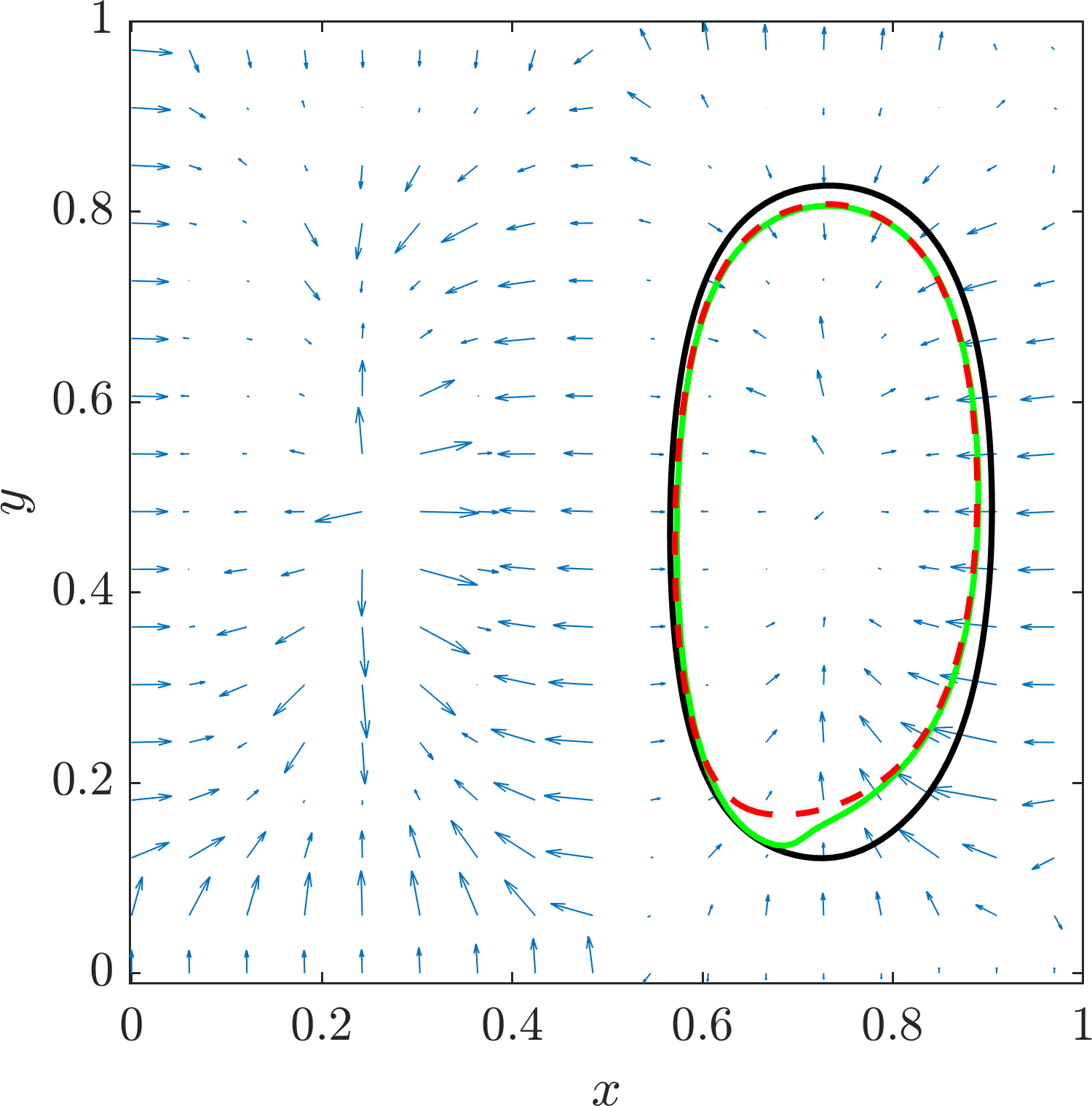}
 \caption{Transitory double gyre: velocity field $v_{\rm level}$ for the level-set evolution (blue arrows); level sets of $u_\eps$ for $\eps=0$ (black)  and $\eps=0.2$ (red) and of $u_0+0.2\dot u_0$, i.e.\ the  prediction of $u_{0.2}$ by linear response (green).}
  \label{fig:rot_double_gyre_velocity}
\end{figure}

Note that while agreeing qualitatively, there are some visible differences  in $\dot u_0$ as computed by the CG vs.\ the TO approach (top row vs.\ bottom row of Figure~\ref{fig:rot_double_gyre_u0}).  The TO method, however, is using considerably less information on the dynamics than the CG approach: the only dynamical data used in TO are the images of the $10^4$ grid nodes at the final time and their derivatives with respect to $\eps$.  In the CG approach, we have to time-integrate the variational equation for each quadrature node in each element of the triangulation which here amounts to $\approx10^5$ time integrations.  If one chooses a correspondingly finer grid for TO -- so that the numerical effort is comparable to CG -- the figures and the prediction on how the coherent set changes become visually indistinguishable.

Using either of the latter approaches, we identify coherent sets in the time frame at $t_0$ by selecting the value $c=0.8412$ from a line search of $c\in [0,\max_x u_0(x)]$ that minimises the dynamic Cheeger value in \eqref{cheeger} with $\Gamma_c=\{x\in \Omega: u_0(x)=c\}$.
Figure \ref{fig:rot_double_gyre_velocity} shows the velocity field $v_{\rm level}$ of the level-set curves at $\eps=0$, which describes how the coherent set boundaries move in the fixed frame at $t_0$ as $t_1$ is extended beyond $t_1=0.6$.
Our linear extrapolation again appears to be accurate, even for a macroscopic extension of the flow time, as the change in the level-set contour from $\eps=0$ (solid line) to $\eps=0.2$ is consistent with the prediction by the velocity field.

\section*{Acknowledgements}

FA is supported by a UNSW PhD scholarship, GF is partially supported by an ARC Discovery Project. OJ acknowledges support by the DFG within the Priority Programme SPP 1881 Turbulent Superstructures.
A visit of FA and GF to the TUM Department of Mathematics was supported by a Universities Australia / DAAD Joint Research Co-operation Scheme, and they thank TUM for hospitality. A visit of OJ to the UNSW School of Mathematics and Statistics was also supported by this scheme and he thanks UNSW for hospitality. 

\bibliographystyle{plain}

\bibliography{lin_resp_dyn}
\end{document}